\documentclass[11pt,a4paper]{article}

\usepackage{mathtools}
\usepackage{authblk} % author
\usepackage{algpseudocode,algorithmicx,algorithm}

\usepackage{mathrsfs}
\usepackage{latexsym,bm}
\usepackage{amsmath,amsfonts,amssymb,amsthm,rotating}
\usepackage{extarrows}

\usepackage{graphicx,subfigure,epstopdf,float}
\usepackage{enumerate,cases,multirow}
\usepackage{makecell}
\usepackage{caption}
\usepackage{qcircuit}

\usepackage{longtable,colortbl,arydshln,threeparttable}
\definecolor{mygray}{gray}{.9}

\usepackage{indentfirst}
\usepackage[top=25mm,bottom=20mm,left=25mm,right=20mm]{geometry}
\usepackage{cite}

\usepackage{listings}

\usepackage{makeidx}        % generate an index, automatically
\usepackage{booktabs}
\usepackage{xcolor}
\usepackage[bookmarksnumbered,colorlinks=true,citecolor=red,linkcolor=red,hyperindex,linktocpage=true]{hyperref}

\newcommand{\ket}[1]{| #1 \rangle} % |u>
\newcommand{\bra}[1]{\langle #1 |} % <u|

\newcommand{\bb}{\boldsymbol}

\def \d {\mathrm{d}}
\def \e {\mathrm{e}}
\def \i {\mathrm{i}}

\newcommand{\bbR}{\mathbb{R}}
\newcommand{\Cp}{\mathcal{P}}

\newcounter{parentalgorithm}

\makeatother

\newtheorem{theorem}{Theorem}[section]
\newtheorem{lemma}{Lemma}[section]

\theoremstyle{remark}
\newtheorem{remark}{\bf Remark}[section]

\numberwithin{equation}{section}

\begin{document}

\title{On the  Schr\"odingerization method for linear non-unitary dynamics with optimal dependence on matrix queries}

\author[1]{Shi Jin\thanks{shijin-m@sjtu.edu.cn}}
\author[1, 2]{Nana Liu\thanks{nana.liu@quantumlah.org} }
\author[3,4,5]{Chuwen Ma\thanks{cwma@math.ecnu.edu.cn}}
\author[6]{Yizhe Peng\thanks{pengyizhe@smail.xtu.edu.cn}}
\author[6]{Yue Yu\thanks{terenceyuyue@xtu.edu.cn}}
\affil[1]{School of Mathematical Sciences, Institute of Natural Sciences, MOE-LSC, Shanghai Jiao Tong University, Shanghai, 200240, China}
\affil[2]{University of Michigan-Shanghai Jiao Tong University Joint Institute, Shanghai 200240, China}
\affil[3]{School of Mathematical Sciences, East China Normal University, Shanghai 200241, China}
\affil[4]{Key Laboratory of MEA, Ministry of Education, East China Normal University, Shanghai 200241, China,
}
\affil[5]{
Shanghai Key Laboratory of PMMP, East China Normal University, Shanghai 200241, China}
\affil[6]{Hunan Key Laboratory for Computation and Simulation in Science and Engineering, Key Laboratory of Intelligent Computing and Information Processing of Ministry of Education, National Center for Applied Mathematics in Hunan, School of Mathematics and Computational Science, Xiangtan University, Xiangtan, Hunan 411105, China}

\maketitle

\begin{abstract}
  The Schr\"odingerization method converts linear partial and ordinary differential equations with non-unitary dynamics into systems of Schr\"odinger-type equations with unitary evolution. It does so via the so-called warped phase transformation that maps the original equation into a Schr\"odinger-type equation in one higher dimension \cite{Schrshort,JLY22SchrLong}. The original proposal used a particular initial function in the auxiliary space that did not achieve optimal scaling in precision. Here we show that, by choosing smoother initial functions in auxiliary space, Schr\"odingerization \textit{can} in fact achieve near optimal and even optimal scaling in matrix queries. We construct three necessary criteria that the initial auxiliary state must satisfy to achieve optimality. %the initialization of the warped phase transformation \cite{JLM24SchrBackward},
  This paper presents detailed implementation of four smooth initializations for the Schr\"odingerization method:
(a) the error function and related functions, (b) the cut-off function, (c) the higher-order polynomial interpolation, and (d) Fourier transform methods. Method (a) achieves optimality and methods (b), (c) and (d) can achieve near-optimality. 
A detailed analysis of key parameters affecting time complexity is conducted.
\end{abstract}

%\textbf{Keywords}: Sch\"odingerization, Quantum circuits, Physical boundary conditions, Linear combination of unitaries, Autonomous system, Shift operators.

\tableofcontents

\section{Introduction}
Quantum computing, an emerging technology, utilizes the principles of quantum mechanics to achieve unprecedented computational power~\cite{hidary2019quantum,NC2010Quantum,preskill2018quantum,rieffel2011quantum}. Quantum algorithms operate within an $n$-qubit Hilbert space of dimension $2^n$, potentially offering polynomial to even exponential computational advantage for models involving vast amounts of data. Hence, it has become an attractive computational paradigm to handle large-scale scientific computing problems that are bottlenecks for classical computation.
A natural application is partial differential equations (PDEs) from time-dependent Schr\"odinger equations, which follow unitary evolutions and hence the wave functions can be coherently represented on quantum computers. Known as Hamiltonian simulations, a variety of efficient algorithms have been developed toward this goal \cite{BCC15,Berry-Childs-Kothari-2015,BCC17,LC17,Low2019qubitization,CGJ19,Berry2019Dyson,
AnLin2021TimeHamiltonian,JLY22SchrLong,An2022timedependent,fang2023time}.

However, many physical phenomena~--~such as combustion, atmospheric and oceanic circulation, and electromagnetic wave propagation with physical boundaries~--~exhibit non-unitary dynamics. Even the time-dependent Schr\"odinger equation becomes non-unitary when artificial boundary conditions are introduced, making traditional Hamiltonian simulation techniques inapplicable~\cite{JLLY23ABC}.
Alternatively, crafting quantum PDE solvers involves discretizing spatial variables to formulate a system of ordinary differential equations (ODEs), which can then be tackled using quantum ODE solvers \cite{Berry-2014,BerryChilds2017ODE,Childs-Liu-2020}.
In recent years, significant progress has been made in designing and analyzing efficient quantum algorithms for linear ODEs. These algorithms can be classified into several categories. The first category involves three steps: discretizing the time variable, encoding the discretized linear differential equation into an enlarged linear system, and solving the resulting system using Quantum Linear System Algorithms (QLSA) \cite{Berry-2014, BerryChilds2017ODE, KroviODE, JLY2022multiscale, WL24, DLX25, Childs-Liu-2020, Berry2024Dyson}. The second category leverages an integral representation of the non-unitary evolution operator, followed by Linear Combination of Hamiltonian Simulations \cite{ALL2023LCH, ACL2023LCH2}, and its relationship to Schr\"odingerization, the method in the third category, is explored in \cite{Li25}. The third category involves dilating the system into a unitary system.
Crucially, if the solution operator for the resulting ODE system is unitary, quantum simulations can achieve reduced time complexity compared to other quantum linear algebra methods \cite{Berry-2014,JLY2022multiscale,ALWZ25,WL24,DLX25}.
In cases where the system is not unitary~--~such as when incorporating physical boundary conditions~--~it becomes necessary to ``dilate'' it to
a unitary system \cite{JLY22SchrLong,Javier2022optionprice,Burgarth23dilations,CJL23TimeSchr,DLX25}. Other methods, such as those employing block-encoding techniques, can be found in \cite{ALWZ25,fang2023time}.

Among the unitarization techniques, the {\it Schr\"odingerization} method proposed in~\cite{Schrshort, JLY22SchrLong} offers a simple and general framework enabling quantum simulation for {\it all} linear PDEs and ODEs. It employs a warped phase transformation to lift the original equations into a higher-dimensional space, where they become Schr\"odinger-type equations~--~with unitary evolutions~--~in the Fourier domain!
This method has since been extended to a wide range of settings, including open quantum systems with non-unitary artificial boundary conditions~\cite{JLLY23ABC}, systems with physical boundary and interface conditions~\cite{JLLY2024boundary}, Maxwell's equations~\cite{JLC23Maxwell,MJLWZ24}, the Fokker-Planck equation~\cite{JLY24FokkerPlanck}, ill-posed problems such as the backward heat equation~\cite{JLM24SchrBackward}. It has also been applied in iterative linear algebra solvers~\cite{JinLiu-LA}. Moreover, as a naturally continuous-variable method~\cite{analogPDE}, it represents the {\it only} viable approach so far for analog quantum simulation of PDEs and ODEs. The method can also be adapted to parabolic PDEs using a Jaynes-Cummings-type model, which is more readily available on current devices ~\cite{JL24JaynesCummings}.

We consider linear dynamical systems with a general evolution operator $ A(t) $ and an inhomogeneous term $ \bm{b}(t) $ as given in \eqref{ODElinear}, and present a detailed implementation of the corresponding quantum algorithm using block-encoding techniques as described in \cite{Low2019Interaction,ACL2023LCH2}. In our implementation, we first transform the system into a homogeneous one by enlarging the system with an auxiliary variable.
Using the Schr\"odingerization method, we then transform it into a Hamiltonian system with unitary evolution operator, which can be efficiently solved on quantum computers. For time-dependent Hamiltonian dynamics, we apply the quantum simulation technique from \cite{CJL23TimeSchr,CJL24} for non-autonomous systems. This technique involves transforming a non-autonomous system into an autonomous one in a higher dimension, avoiding the need for the complicated Dyson series. Consequently, we focus exclusively on analyzing the optimal dependence in the time-independent case.

In this article, we focus on the {\it optimal scaling} behavior of the Schr\"odingerization method for non-unitary dynamics. In the original Schr\"odingerization method, a simple even-function $\psi(p) = \e^{-|p|}$ is used as the initial auxiliary state, resulting only in a first-order approximation, due to the lack of regularity of this state. This lack of regularity meant that achieving precision $\epsilon$ may require a small enough mesh size $\triangle p = \mathcal{O}(\epsilon)$, which results in the maximum absolute value among the discrete Fourier modes scaling being $\mathcal{O}(1/\epsilon)$, i.e., $\mu_{\max} = \mathcal{O}(1/\epsilon)$. This is not optimal because the query complexity linearly depends on $\mu_{\max}$, which is the maximum Fourier mode in absolute value, for the extended variable $ p $.

Since this non-optimal $\mathcal{O}(1/\epsilon)$ scaling is ultimately due to the lack of regularity of the initial auxiliary state function, then improving this scaling is only a question of how to smooth out this initial function in an appropriate way. For example, as already pointed out in \cite{JLM24SchrBackward}, we can achieve improved~--~near optimal~--~scaling by employing a smoother initialization. In this article we provide four different smooth initializations, with detailed analysis on their complexities. One of these methods (Section~\ref{sec:optimal scaling}) leads to \textit{optimal complexity} and the other three methods (Section~\ref{sec:smoothmethod}) lead to \textit{near optimal complexity}. In Section~\ref{sec:optimal scaling} we also provide three necessary criteria that $\psi(p)$ should satisfy in order to achieve optimal scaling.

Through a detailed analysis of the parameters affecting time complexity, we find that the query complexity scales linearly with $ \mu_{\max} $, the maximum Fourier mode in absolute value, for the extended variable $ p $ with a generic initialization function $\psi(p)$ in the $p$-domain.
Assuming that $\| \psi - P_h \psi \|_{L^2(\mathbb{R})} \le \epsilon$ for $ \psi(p) \in H^r(\mathbb{R})$, with $P_h \psi$ being the discrete Fourier approximation of $\psi$, we observe that $\mu_{\max}$ scales as $\mathcal{O}((1/\epsilon)^{1/r} \|\psi^{(r)}\|_{L^2(\mathbb{R})}^{1/r})$, where $\psi^{(r)}$ denotes the $r$-th derivative of $\psi$ and $H^r$ is the standard Sobolev space. When $r$ is sufficiently large, specifically $r \simeq \log(1/\epsilon)$, we have $\mu_{\max} \simeq \|\psi^{(r)}\|_{L^2(\mathbb{R})}^{1/r}$. Thus, the precision scaling problem reduces to identifying an appropriate function $\psi(p)$ such that
\begin{equation}\label{psir}
\|\psi^{(r)}\|_{L^2(\mathbb{R})}^{1/r} \le C r^{1/\beta}, \qquad \beta\in (0,1],
\end{equation}
where $\beta = 1$ implies the optimal precision dependence. 

Based on this observation, we provide an abstract framework for the complexity analysis in Theorem \ref{thm:complexity}.
For academic interest, we first provide several sufficiently smooth initializations that offer nearly exponential speedup in the $p$-variable over the original Schr\"odingerization method in terms of precision $\varepsilon$, including cut-off functions, higher-order interpolation, and Fourier transform methods.
Since we impose the condition $\psi(p) = \e^{-p}$ for $p>0$, the dependence on matrix queries can only be made near-optimal, as
$\beta<1$. Achieving optimal dependence requires $\beta=1$. We therefore summarize the conditions on $\psi$ necessary for attaining optimal dependence and provide a detailed estimate. In particular, we employ the error function $\text{erf}(p)$ to achieve  the required optimal bound and show that this is not the unique function, but is the simplest.
%To achieve optimal dependence with $\beta=1$, we relax the above condition to construct an approximate version using the error function $\text{erf}(p)$ and explicitly establish the required bound.

Our construction is based on a simple yet important observation: the process of constructing a smooth initialization function $\psi(p)$ reduces to finding a smooth approximation of the step function. Such smoothness is typically achieved through convolution.
Given the periodic boundary conditions required for the discrete Fourier transform, we first consider using a mollifier as the convolution kernel. The mollifier is infinitely smooth and has compact support, thereby realizing the {\it exact} periodic boundaries. However, upon careful analysis, we find that smoothness alone is insufficient to satisfy \eqref{psir}. In fact, we require a stricter condition: the step function approximation must also be {\it analytic}. Under this condition, the most natural choice for the convolution kernel is the Gaussian $\e^{-p^2}$, which leads to the expression involving the error function.
For both choices of convolution kernels, we provide a rigorous proof of the bounds for the derivatives of the smooth initialization functions.
This demonstrates that the mollifier yields $\beta = 1/2$, while the Gaussian results in $\beta = 1$. Consequently, we derive the near-optimal cost for the mollifier and the optimal scaling behavior for the Gaussian by applying the abstract framework for complexity analysis.

In Table \ref{tab:comparison}, we compare our algorithm with previous approaches in the homogeneous case.
It is evident that the Schr\"odingerization method, with sufficiently smooth initialization in $p$,  achieves both optimal state preparation cost and optimal dependence of the number of queries to the matrix on all parameters. We also note that the improvement in the LCHS method in \cite{ACL2023LCH2} only leads to a sub-optimal dependence on matrix queries, as $ \beta $ cannot be exactly equal to 1 (i.e., $ 1/\beta > 1 $), where $\beta$ is the parameter in the kernel function of the improved LCHS.

During the revision stage of this work, we became aware of a concurrent study by \cite{Low25LCHS}, which presents a generalized {\it approximate} version of LCHS. This extension incorporates a kernel function with exponential decay, enabling a quantum ODE algorithm that achieves optimal dependence on precision.

\begin{table}[!htb]
  \centering
  \caption{Comparison among improved Schr\"odingerization and previous methods for homogeneous dynamical systems $\d \bm{u}/\d t= A\bm{u}$. Here, $u_r = \frac{\|\bm{u}_0\|}{\|\bm{u}(T)\|}$, $\alpha_A \ge \|A\|$, $T$ is the evolution time, $\varepsilon$ is the error, and $\beta\in(0,1)$. All but the spectral method assume the real part
of $A$ to be negative semi-definite, while in the spectral method $A$ is assumed to be diagonalizable with
matrix  $V$ such that $\kappa_V \ge \|V^{-1}\| \|V\|$ and  all the eigenvalues of $A$ have non-positive real parts.}\label{tab:comparison}
  \begin{tabular}{c|c|cccc}
  \hline\hline
  Method & Queries to $A$   &Queries to $\bm{u}_0$  \\
  \hline
   Spectral method \cite{Childs-Liu-2020}     & $\widetilde{\mathcal{O}}(u_r \kappa_V \alpha_A T \text{poly}(\log\frac{1}{\varepsilon}))$ & $\widetilde{\mathcal{O}}(u_r \kappa_V \alpha_A T \text{poly}(\log\frac{1}{\varepsilon}))$\\
   \hline
   Truncated Dyson \cite{Berry2019Dyson}     & $\widetilde{\mathcal{O}}(u_r \alpha_A T (\log\frac{1}{\varepsilon})^2)$  & ${\mathcal{O}}(u_r \alpha_A T \log\frac{1}{\varepsilon})$ \\
   \hline
   Time-marching \cite{fang2023time}        &  $\widetilde{\mathcal{O}}(u_r\alpha_A^2 T^2 \log\frac{1}{\varepsilon})$ & ${\mathcal{O}}(u_r )$\\
   \hline
   \shortstack{Improved LCHS \\ time-dependent \cite{ACL2023LCH2}}       & $\widetilde{\mathcal{O}}(u_r  \alpha_A T (\log\frac{1}{\varepsilon})^{1+1/\beta})$ & ${\mathcal{O}}(u_r  )$\\
   \hline
   \shortstack{Improved LCHS \\ time-independent \cite{ACL2023LCH2}}        & $\widetilde{\mathcal{O}}(u_r  \alpha_A T (\log\frac{1}{\varepsilon})^{1/\beta})$  & ${\mathcal{O}}(u_r  )$\\
   \hline
    \shortstack{Optimal LCHS \\ time-independent \cite{Low25LCHS}}  & $\widetilde{\mathcal{O}}(u_r  \alpha_A T \log\frac{1}{\varepsilon})$ & ${\mathcal{O}}(u_r  )$\\
    \hline
   This work, time-dependent           & $\widetilde{\mathcal{O}}(u_r  \alpha_A T (\log\frac{1}{\varepsilon})^2)$ & ${\mathcal{O}}(u_r  )$\\
   \hline
   This work, time-independent          & $\widetilde{\mathcal{O}}(u_r  \alpha_A T \log\frac{1}{\varepsilon})$  & ${\mathcal{O}}(u_r  )$\\
  \hline \hline
\end{tabular}
\end{table}

\paragraph{Notation.}
Throughout the paper, we adopt zero-based indexing: indices
$ j \in \{0,1,\dots,N-1\}, $
and we also write \(j \in [N]\) with \([N] := \{0,1,\dots,N-1\}\). We use
\(\ket{j} \in \mathbb{C}^{N}\) to denote the \(j\)-th canonical basis vector,
whose \(j\)-th component is \(1\) and all other components are \(0\).
We denote the identity and zero matrices by \(I\) and \(O\), respectively; their
dimensions should be clear from the context. When clarification is needed, we
write \(I_N\) for the \(N \times N\) identity matrix. In particular,
\(\mathbf{1}\) denotes the \(2 \times 2\) identity matrix (acting on a single
qubit).

Vector-valued quantities are denoted by boldface symbols, e.g.\ \(\bb{u}\).
Given a nonzero vector \(\bb{u}\), the notation \(\ket{\bb{u}}\) represents the pure quantum state obtained by normalizing \(\bb{u}\) in the Euclidean norm,
$ \ket{\bb{u}} \;=\; \frac{\bb{u}}{\|\bb{u}\|}.$

Unless otherwise specified, for a vector \(\bb{u} \in \mathbb{C}^N\) we write
$ \|\bb{u}\| := \|\bb{u}\|_2 $
for the standard Euclidean ($l^2$) norm. For a matrix
\(A \in \mathbb{C}^{N \times N}\), the notation \(\|A\|\) refers to the operator
norm induced by the Euclidean norm,
$ \|A\| := \sup_{\bb{u} \neq 0} \frac{\|A \bb{u}\|}{\|\bb{u}\|}. $
When Sobolev norms are used, we write \(\|\cdot\|_{H^s(\Omega)}\),
\(\|\cdot\|_{L^2(\Omega)}\), etc., and these always denote the standard norms on
the corresponding Sobolev spaces.

% Throughout the paper, we use a 0-based indexing, i.e. $j= \{0,1,\cdots,N-1\}$, or $j\in [N]$, and  $\ket{j}\in  \mathbb{C}^{N\times 1}$, to denote a vector with the $j$-th component being $1$ and others $0$.
% We shall denote the identity matrix and null matrix by $I$ and $O$, respectively,
% and the dimensions of these matrices should be clear from the context. Otherwise,
% The notation $I_N$ stands for the $N$-dimensional identity matrix, and $\textbf{1}$ denotes $2$-dimension identity matrix.
% Moreover, the vector-valued quantities are denoted by boldface symbols, e.g. $\bb{u}$,
% and $\ket{\bb{u}}$ represents a (pure) quantum state that is the normalized vector under $2$-norm
% $\ket{\bb{u}} = \bb{u}/\|\bb{u}\|$.

For asymptotic estimates we write $O(\cdot)$, $\Omega(\cdot)$ and $\Theta(\cdot)$
in the usual sense. In particular, $f=O(g)$ means that
$|f|\le C\,g$ for some constant $C>0$ independent of the relevant parameters.
We use $\widetilde{O}(\cdot)$ to suppress polylogarithmic factors, e.g.\
$ f = \widetilde{O}(g)$  means $ f = O\bigl(g\,\mathrm{polylog}(g)\bigr). $
We also write $f\lesssim g$ to indicate an inequality of the form
$
f \le C\,g,
$
where the constant $C>0$ is independent of the mesh size $h$, the final time
$T$, the target accuracy $\varepsilon$, and other sensitive problem parameters.
Unless stated otherwise, all logarithms are natural logarithms.

\paragraph{Organization of the paper.} The paper is structured as follows.
In Section \ref{sec:Schrodingerization}, we offer an overview of the Schr\"odingerization approach, present the full implementation by using block-encoding techniques and establish an abstract framework for the complexity analysis.
Section \ref{sec:smoothmethod} provide sufficiently smooth initializations that offer nearly exponential speedup in the $p$-variable over the original Schr\"odingerization method. Section \ref{sec:optimal scaling} demonstrates how optimal scaling in matrix queries can be attained through the modification of the initializations of the warped phase transformation. We establish the optimal dependence by constructing a function using the error function $\text{erf}(p)$. Section \ref{sec:err} shows the detailed error estimate for the Schr\"odingerization.
Finally, some discussions are presented in the last section.

\section{The Schr\"odingerization method for non-unitary dynamics} \label{sec:Schrodingerization}

Consider a system of linear dynamical system in the form
\begin{equation}\label{ODElinear}
 \begin{cases}
 \dfrac{\d }{\d t} \bb{u}(t) = A(t) \bb{u}(t) + \bb{b}(t),  \qquad t\in (0,T), \\
 \bb{u}(0) = \bb{u}_0,
 \end{cases}
\end{equation}
where $T$ is the evolution time, $\bb{u}= [u_0, u_1,\cdots, u_{N-1}]^\top, \bb{b}= [b_0, b_1,\cdots, b_{N-1}]^\top \in \mathbb{C}^N$  and $A \in \mathbb{C}^{N\times N} $. In general, $A$ is not anti-Hermitian, i.e.,  $A^{\dagger} \neq -A$, where "$\dagger$" denotes conjugate transpose. When $A$
is a linear operator, \eqref{ODElinear} is a system of ODEs. When $A$
is a linear differential operator, \eqref{ODElinear} is a system of PDEs.
By introducing an auxiliary vector function $\bm{r}(t)$ that remains constant in time if $\bm{b}\neq 0$, system \eqref{ODElinear} can be rewritten as a homogeneous system
\begin{equation}\label{eq: homo Au}
	\frac{\d}{\d t} \bm{u}_f
	= A_f \bm{u}_f, \quad
	A_f  =  \begin{bmatrix}
		A &  B \\
		O &O
	\end{bmatrix},\quad
	\bm{u}_f(0) =\bm{u}_{I}:= \begin{bmatrix}
		\bm{u}_0 \\
		\bm{r}_0
	\end{bmatrix},
\end{equation}
where $B=\text{diag}\{b_0/\gamma_0, \cdots, b_{N-1}/\gamma_{N-1}\}$ and $\bb{r}_0 = [\gamma_0, \cdots , \gamma_{N-1}]^\top$, with
\begin{equation}\label{eq:gamma}
	\gamma_i = T \sup_{t\in [0,T]} |b_i(t)|, \qquad i = 0,1,\cdots,N-1.
\end{equation}
Here, each $\sup_{t\in [0,T]} |b_i(t)|$ can be replaced by its upper bound and we set $b_i/\gamma_i = 0$ if $b_i(t) \equiv 0$.

\subsection{The Schr\"odingerizaton method}

In this section, we briefly review the Schr\"odingerization approach for general linear dynamical systems.
For a general $A_f$, we first decompose $A_f$ into a Hermitian term and an anti-Hermitian term:
\[A_f(t) = H_1(t) + \i H_2(t), \qquad \i = \sqrt{-1},\]
where
\[
H_1(t) = \frac{A_f(t)+A_f^{\dagger}(t)}{2} =  \begin{bmatrix}
	H_1^A & \frac{ B}{2}\\
	\frac{B^{\top}}{2} &O
\end{bmatrix}, \quad H_2(t) = \frac{A_f(t)-A_f^{\dagger}(t)}{2 \i} = \begin{bmatrix}
H_2^A & \frac{B}{2\i}\\
-\frac{ B^{\top}}{2\i} &O
\end{bmatrix},
\]
with $H_1^A = (A+A^{\dagger})/2$ and $H_2^A = (A-A^{\dagger})/(2\i)$.
Throughout the article, we assume that the real part matrix $ H_1^A $ is negative semi-definite. More general cases are addressed in \cite{JLM24SchrInhom, JLM24SchrBackward, JLY24FokkerPlanck}.

Using the warped phase transformation $\bb{w}(t,p) = \e^{-p} \bb{u}_f(t)$ for $p\ge 0$ and symmetrically extending the initial data to $p<0$,  system \eqref{ODElinear} is  then transformed to a system of linear convection equations \cite{Schrshort,JLY22SchrLong}:
\begin{equation}\label{u2v}
\begin{cases}
 \dfrac{\partial}{\partial t} \bb{w}(t,p)  = - H_1(t) \partial_p \bb{w} + \i H_2(t) \bb{w}, \\
 \bb{w}(0,p) = \psi(p) \bb{u}_{I},
 \end{cases}
\end{equation}
where $\psi(p):=\e^{-|p|}$.
According to \cite[Theorem 3.1]{JLM24SchrInhom}, we can restore the solution $\bb{u}_f(t)$ by
\begin{equation}\label{eq:recovery}
\bm{u}_f = \e^p \bm{w}(t,p),\quad p\geq p^{\Diamond}   = \lambda_{\max}^{+}(H_1)T.
\end{equation}
Here $\lambda_{\max}^{+}(H_1)$  is defined by
% \begin{equation}\label{eq:lambda_max_plus}
% \lambda_{\max}^{+}(H_1) = \max\bigg\{\sup_{0<t<T} \{|\lambda|: \lambda\in \sigma(H_1(t)), \lambda>0\},0\bigg\}, {\color{red} \mbox{delete 0?}}
% \end{equation}
\begin{equation}\label{eq:lambda_max_plus}
	\lambda_{\max}^{+}(H_1) = \begin{cases}
		\displaystyle \sup_{\substack{\lambda \in \sigma(H_1(t)) \ 0 < t < T, 0<\lambda }} |\lambda|, & \text{if } \exists \lambda > 0 \text{ in } \sigma(H_1(t)) \text{ over } [0,T),\\
		0, & \text{otherwise},
	\end{cases}
\end{equation}
 with $\sigma(H_1)$ the set of eigenvalues of $H_1$. 
 Similarly, $\lambda_{\max}^-(H_1)$ is defined by 
\begin{equation}\label{eq:lambda_max_minus}
	\lambda_{\max}^{-}(H_1) = \begin{cases}
		\displaystyle \sup_{\substack{\lambda \in \sigma(H_1(t)) \ 0 < t < T, \lambda<0 }} |\lambda|, & \text{if } \exists \lambda <0 \text{ in } \sigma(H_1(t)) \text{ over } [0,T),\\
		0, & \text{otherwise}.
	\end{cases}
\end{equation}
 Since $H_1^A$ is negative, it is easy to find from \eqref{eq:gamma} that
 \[\lambda_{\max}^{+}(H_1)T \leq \frac{1}{2}\|B\|_{\max} T\leq \frac{1}{2}.\]

 For numerical implementation, we truncate the extended region to a finite interval $p\in [-L,R]$ with $L>0$ and $R>0$ satisfying
\begin{equation}\label{eq: L,R,criterion}
\e^{-L+\lambda_{\max}(H_1)T}\approx \e^{-R+\lambda_{\max}(H_1)T} \approx \epsilon.
\end{equation}
Here $\lambda_{\max}(H_1)$ denotes the largest absolute value among the eigenvalues of $H_1$, and $\epsilon$ is a predetermined tolerance constant, which will be specified later.

The requirement in \eqref{eq: L,R,criterion} is explained as follows. Since the original problem is posed on the whole space, we truncate it to a finite interval $ [-L, R] $ with periodic boundary conditions. This means that we require $ \bb{w}(0, -L) \approx \bb{w}(0, R) \approx \epsilon $, or equivalently, $ \e^{-L} \approx \e^{-R} \approx \epsilon $. For the transport equation $ u_t - a u_p = 0 $ with $ a > 0 $, the initial value at $ p_0 $, i.e., $ u(0, p_0) $, will remain constant along the characteristic line $ p + at = p_0 $, which implies $ u(t, p_0 - at) = u(0, p_0) $. For the transport equation in \eqref{u2v} with periodic boundaries, the solution values at $ p = -L $ and $ p = R $ must also be compatible along characteristics, based on the initial data in the regions $ (-L, -L + \lambda_{\max}(H_1)T) $ and $ (R - \lambda_{\max}(H_1)T, R) $.
If the initial data in these regions has already decayed to the level of $\epsilon$, then the boundary values and their higher derivatives satisfy
\[\bm{w}^{(k)}(t,-L)\approx \bm{w}^{(k)}(t,R)\approx \epsilon, \quad t\in [0,T],\]
so that the periodic boundary condition is consistent with the infinite-domain problem up to accuracy $\epsilon$.

Toward this end,  we choose a uniform mesh size $\triangle p = (R+L)/N_p$ for the auxiliary variable with $N_p=2^{n_p}$ being an even number, with the grid points denoted by $-L = p_0<p_1<\cdots<p_{N_p} = R$.
 Let the vector $\bb{W}_h \in \mathbb{C}^{N_{np}\times 1}$ with $N_{np} = N\times N_p$ be the collection of $\bm{w}_h(t,p)$ at these grid points, defined more precisely as $
\bb{W}_h(t) = \sum_{k\in [N_p],i\in [N]} w_{i,h}(t,p_k) \ket{k,i}$, where $w_{i,h}$ is the $i$-th entry of $\bb{w}_h$ and
$\ket{k,i} = \ket{k}\ket{i}$.

By applying the discrete Fourier transform in the $p$  direction, one arrives at
\begin{equation}\label{heatww}
\frac{\d}{\d t} \bb{W}_h(t) = -\i (P_\mu \otimes  H_1 ) \bb{W}_h + \i (I\otimes H_2 ) \bb{W}_h , \quad
\bb{W}_h(0) = \bm{\psi} \otimes \bb{u}_{I},
\end{equation}
where $\bb{\psi} = [\psi(p_0), \cdots, \psi(p_{N_p-1})]^{\top}$.
 Here, $P_\mu$ is the matrix expression of the momentum operator $-\i\partial_p$, given by
\begin{equation}\label{Pmu}
P_\mu = \Phi D_\mu \Phi^{-1},  \qquad D_\mu = \text{diag}(\mu_0, \cdots, \mu_{N_p-1}),
\end{equation}
where $\mu_k = \frac{2\pi}{R+L} ( k - \frac{N_p}{2}) $ are the Fourier modes and
\[\Phi = (\phi_{jl})_{N_p \times N_p} = (\phi_l(x_j))_{N_p\times N_p}, \qquad \phi_l(x) = \e ^{\i \mu_l (x+L)} \]
is the matrix representation of the discrete Fourier transform.
At this point, we have successfully mapped the dynamics back to a Hamiltonian system.
By a change of variables $\tilde{\bb{W}}_h = (\Phi^{-1} \otimes I)\bb{W}_h$, one has
\begin{equation}\label{generalSchr}
\frac{\d}{\d t} \tilde{\bb{W}}_h(t) = -\i H(t) \tilde{\bb{W}}_h(t) ,
\end{equation}
where $ H = D_\mu \otimes H_1 -  I \otimes H_2 $.

 \begin{remark}
Our method for solving \eqref{ODElinear} with a time-dependent source term encodes $\bb{b}(t)$ directly within the coefficient matrix. This results in the same query complexity for both the coefficient matrix $A(t)$ and $\bb{b}(t)$. To minimize the repeated use of the state preparation oracle $ O_b $ for the source term $ \bb{b} $, when $\bb{b}(t)$ is time-independent, we can instead consider a simpler enlarged system
\[
\frac{\d }{\d t} \bb{u}_f(t) = \begin{bmatrix}
A  &  \frac{I}{T} \\
O     &  O
\end{bmatrix} \bb{u}_f(t), \qquad \bb{u}_f(t) = \begin{bmatrix} \bb{u}(t) \\ T \bb{b} \end{bmatrix}, \qquad
\bb{u}_f(0) = \bm{u}_I:= \begin{bmatrix} \bb{u}(0) \\ T \bb{b} \end{bmatrix}.
\]
In the time-dependent case, it may be advantageous to separately implement their homogeneous and inhomogeneous parts and combine them using the LCU technique \cite{ACL2023LCH2}. Each execution of the LCU procedure requires $ \mathcal{O}(1) $ uses of the associated preparation oracles, with the overall complexity primarily dependent on the success probability.
 \end{remark}

\subsection{Quantum simulation for time-dependent Schr\"odingerized system}
\label{sub:time-dependent}

If the coefficient in the dynamical system is time-dependent, namely a non-automomous system, one can turn it into an autonomous unitary system via dimension lifting
 \cite{CJL23TimeSchr}.
First, via Schr\"odingerization, one obtains a time-dependent Hamiltonian
\begin{equation}
	\frac{\d}{\d t}  \tilde{\bm{W}}_h = -\i H(t) \tilde{\bm{W}}_h,\quad
	H = H^{\dagger}.
\end{equation}
By introducing a new ``time" variable $s$, the problem becomes a new linear PDE  defined in one higher dimension but with time-independent coefficients,
\begin{equation}
	\frac{\partial \bm{v}}{\partial t} = -\frac{\partial \bm{v}}{\partial s} -\i H(s) \bm{v} \qquad
	\bm{v}(0,s) = \delta(s) \tilde{\bm{w}}_h(0), \quad s\in \bbR,
\end{equation}
where $\delta(s)$ is the dirac $\delta$-function.  One can easily recover $ \tilde{\bm{W}}_h$ by  $\tilde{\bm{W}}_h = \int_{-\infty}^{\infty} \bm{v}(t,s)\;\d s$.

Since $\bb{v}$ decays to zero as $s$ approaches infinity, the $s$-region can be truncated to $[-\pi S, \pi S]$, where $\pi S > 4\omega +T$, with $2\omega$ representing the length of the support set of the approximated delta function. Choosing $S$ sufficiently large ensures that the compact support of the approximated delta function remains entirely within the computational domain throughout the simulation, allowing the spectral method to be applied.
The transformation and difference matrix are defined by
\begin{equation*}
	(\Phi_s)_{lj} = (\e^{\i\mu^s_l(j\triangle s)}), \quad
	D_s = \text{diag} \{\mu^s_0,\mu^s_1,\dots,\mu^s_{N_s-1}\},
	\quad \mu^s_l = (l-\frac{N_s}{2})S, \quad l,j\in[N_s],
\end{equation*}
where $\triangle s = 2\pi S/N_s$.
Applying the discrete Fourier spectral discretization, it yields a time-{\it independent} Hamiltonian system as
\begin{equation}\label{eq:tilde v time independent}
	\frac{\d}{\d t} \tilde{\bm{V}}_h = -\i \big(D_s \otimes I+I_{N_s}\otimes H\big)\tilde{\bm{V}}_h, \quad
	\tilde{\bm{V}}_h (0) = [\Phi_s^{-1}\otimes I] (\mathbf{\delta}_{\bm{h}}\otimes
	\tilde{\bm{W}}_h(0)),
\end{equation}
where $\mathbf{\delta}_{\bm{h}} = \sum\limits_{j\in [N_s]}\delta_w(s_j)\ket{j}$ with
$s_j = -\pi S + j\triangle s$ and
$\delta_{\omega}$ is an approximation to $\delta$ function defined, for example, by choosing
\begin{equation*}
	\delta_{\omega}(x) = \frac{1}{\omega} \bigg(1-\frac{1}{2}|1+\cos(\pi \frac{x}{\omega})|\bigg)\quad |x|\leq \omega, \quad
	\delta_{\omega}(x) = 0\quad |x|\geq \omega.
\end{equation*}
Here $\omega = m\triangle s$, where $m$ is the number of mesh points within the support of $\delta_{\omega}$.

\subsection{Hamiltonian system for quantum computing}
As discussed in Section \ref{sub:time-dependent}, a time-dependent system can be transformed into a time-independent system by adding an additional dimension. Therefore, the subsequent discussion will focus exclusively on the time-independent case. For further details on time-dependent systems, we refer to \cite{CJL24}.

From \eqref{generalSchr}, a quantum simulation can be carried out on the Hamiltonian system above:
\begin{equation*}
	\ket{\bm{W}_h(T)} = \big[\Phi \otimes I\big]\cdot
	\mathcal{U}(T) \cdot
	\big [ \Phi ^{-1}\otimes I\big] \ket{\bm{W}_h(0)},
\end{equation*}
where $\mathcal{U}(T) = \e^{-\i H T}$ is a unitary operator, and $\Phi $ (or $\Phi^{-1}$) is completed by (inverse) quantum Fourier transform (QFT or IQFT).
The complete circuit for implementing the quantum simulation of $\ket{\bm{w}_h}$ is illustrated in Fig.~\ref{schr_circuit}.
\begin{figure}[!htb]
	\centering
	\centerline{
		\Qcircuit @C=1em @R=2em {
			\lstick{\hbox to 2.7em{$\ket{\bm{\psi}_h}$\hss}}
			& \qw
			& \gate{\text{IQFT}}
			& 	\qw
			& \multigate{1}{\mathcal{U}(T)}	
			&  \qw
			& \gate{\text{QFT}}
			& \qw	
			& \qw  &\meterB{\ket{k}}\\
			\lstick{\hbox to 2.7em{$\ket{\bm{u}_f(0)}$\hss}}
			& \qw
			& \qw
			& \qw
			& \ghost{\mathcal{U}(T)}
			& \qw
			& \qw
			& \qw
			& \qw  & \hbox to 2em{$\ket{\bm{u}_f(T)}$\hss}
		}}
	\caption{Quantum circuit for Schr\"odingerization of ~\eqref{generalSchr}, where $\bm{\psi}_h = \sum_{k\in [N_p]} \psi(p_k)\ket{k} $.}
	\label{schr_circuit}
\end{figure}

From \eqref{eq:recovery}, one can recover the target variables for $\bm{u}_f$ by performing a measurement in the computational basis:
\[M_k = \ket{k}\bra{k} \otimes I, \quad k \in \{j: p_j\geq p^{\Diamond} \;\text{and}\; p_j=\mathcal{O}(1)\}=:I_\Diamond,\]
where $I_\Diamond$ is referred to as the recovery index set.
The state vector is then collapsed to
\[ \ket{\bm{w}_*} \equiv \ket{k_*} \otimes \frac{1}{\mathcal{N}}\Big(\sum\limits_i w_{k_*i} \ket{i} \Big) , \quad
\mathcal{N} = \Big(\sum\limits_i |w_{k_* i}|^2 \Big)^{1/2},\]
where $w_{k_*,i} = \bra{k_*}\bra{i} \otimes \bb{W}_h$
for some $k_*$ in the recovery index set $I_\Diamond$ with the probability
\begin{align*}
	\text{P}_{\text{r}}(\bb{w}(T,p_{k_*}))
	= \frac{\sum_i |w_{k_*i}(T)|^2}{\sum_{k,i} |w_{ki}(T)|^2}
	= \frac{\|\bm{w_*}(T)\|^2}{\|\bb{w}_h(T)\|^2}
	=\frac{\|\bm{w_*}(T)\|^2}{\|\bb{w}_h(0)\|^2}.
	%= \frac{\sum_i |w_{i,h}(T,p_{k_*})|^2}{\sum_{k,i} |w_{i,h}(T,p_k)|^2}
	%= \frac{ \|\bm{w}_h(T,p_{k_*})\|^2}{\sum_k \|\bm{w}_h(T,p_k)\|^2}.
\end{align*}
Then the likelihood of acquiring $\ket{\bm{w}_*}$ that satisfies $k_* \in I_\Diamond$ is given by
\begin{equation}\label{Prw}
	\text{P}_{\text{r}}(\bb{w}) = \sum_{k\in I_\Diamond} \text{P}_{\text{r}}(\bb{w}(T,p_{k}))
	= \frac{ \sum_{k\in I_\Diamond}\sum_{i\in [N]}|w_{k,i}(T)|^2}{\|\bb{w}_h(0)\|^2}
	= \frac{C_{e0}^2}{C_e^2}\frac{\|\bb{u}_f(T)\|^2}{\| \bb{u}_I \|^2},
\end{equation}
where
\begin{equation}\label{Ce0}
	C_{e0} = \Big(\sum_{p_k\in I_{\Diamond}}  (\psi(p_k))^2  \Big)^{1/2}, \qquad C_e = \Big(\sum_{k=0}^{N_p-1} (\psi(p_k))^2  \Big)^{1/2}, \qquad \psi(p) = \e^{-|p|}.
\end{equation}
If $N_p$ is sufficiently large, we have
\[\triangle p C_{e0}^2 \approx  \int_{p^\Diamond}^{\infty} \e^{-2p}   \d p
= \frac12 \e^{-2p^\Diamond} , \qquad
\triangle p C_e^2 \approx  \int_{-\infty} ^{\infty}  \e^{-2p} \d p= 1,\]
where $\triangle p = (R+L)/N_p$, then  it yields
\begin{equation}\label{cece0}
	\frac{C_{e0}^2}{C_e^2} \approx \frac12 \e^{-2p^\Diamond} \ge  \frac{1}{2}\e^{-1}.
\end{equation}
Since $\bm{u}_f(t) = \begin{bmatrix} \bm{u}(t) \\ \bm{r}_0 \end{bmatrix}$, one can perform a projection to get $\ket{\bm{u}(T)}$ with the probability $\dfrac{\|\bm{u}(T)\|^2}{\|\bm{u}_f(T)\|^2}$.
The overall probability of retrieving $\bb{u}$ is then approximated by
\[\text{P}_{\text{r}}(\bb{u}) =
\text{P}_{\text{r}}(\bb{w}) \cdot \frac{\|\bb{u}(T)\|^2}{\|\bb{u}_f(T)\|^2}
=
 \frac{C_{e0}^2}{C_e^2} \frac{\|\bm{u}(T)\|^2}{\|\bm{u}_I\|^2}
= \frac{C_{e0}^2}{C_e^2}\frac{\|\bm{u}(T)\|^2}{\|\bm{u}_0\|^2+T^2 \|\bm{b}\|_{\mathrm{smax}}^2},\]
where
\begin{equation}\label{bmax}
\|\bm{b}\|_{\mathrm{smax}}^2 = \sum_{i=0}^{N-1} \Big(\sup_{t\in [0,T]} |b_i(t)| \Big)^2.
\end{equation}

By using the amplitude amplification, the repeated times for the measurements can be approximated as
\begin{equation}\label{gtimes}
	g = \mathcal{O}\Big(\frac{C_e}{C_{e0}} \frac{\|\bm{u}_I\|} {\|\bm{u}(T)\|} \Big) = \mathcal{O}\Big(\frac{\|\bm{u}_0\|+T \|\bm{b}\|_{\mathrm{smax}}} {\|\bm{u}(T)\|}\Big).
\end{equation}
The quantity $g$ in \eqref{gtimes} is comparable to the number of repeated times by directly projecting onto $\ket{k_*}\ket{0}$ for $k_*\in I_{\Diamond}$.

\subsection{Detailed implementation of the Hamiltonian simulation}

For the Hamiltonian simulation of $ \mathcal{U}(T) = \e^{-\i H T} $, where $H$ arises from the time-dependent system, one can apply established quantum algorithms from the literature. For instance, Hamiltonian simulation with nearly optimal dependence on all parameters is discussed in \cite{Berry-Childs-Kothari-2015}, where sparse access to the Hamiltonian $ H $ is assumed.

%To circumvent the expense of encoding $ D_p $,
One can express the evolution operator $ \mathcal{U}(T) $ as a select oracle
\begin{equation*}\label{selectVk}
	\mathcal{U}(T) = \sum_{k=0}^{N_p-1}  \ket{k}\bra{k} \otimes \e^{-\i (\mu_k H_1- H_2) T}
	=: \sum_{k=0}^{N_p-1} \ket{k}\bra{k} \otimes V_k(T).
\end{equation*}
Since the unitary $V_k(T)$ corresponds to the simulation of the Hamiltonian $H_{\mu_k}:=\mu_k H_1 - H_2$, we assume the block-encoding oracles encoding the real and imaginary parts separately, namely
\[(\bra{0}_a\otimes I)U_{H_i} (\ket{0}_a\otimes I) = \frac{H_i}{\alpha_i}, \qquad i = 1,2, \]
where $\alpha_i\ge \|H_i\|$ is the block-encoding factor for $i = 1,2$.

According to the discussion in \cite[Section 4.2.1]{ACL2023LCH2}, there is an oracle $\text{HAM-T}_{H_{\mu}}$ such that
\begin{equation}\label{HAMT}
	(\bra{0}_{a'} \otimes I) \text{HAM-T}_{H_\mu}(\ket{0}_{a'} \otimes I)
	= \sum_{k=0}^{N_p-1} \ket{k}\bra{k} \otimes \frac{H_{\mu_k}}{\alpha_1 \mu_{\max} + \alpha_2},
\end{equation}
where $H_{\mu_k} = \mu_k H_1 - H_2$ and $\mu_{\max} = \max_k |\mu_k|$ represent the maximum absolute value among the discrete Fourier modes. This oracle only uses $\mathcal{O}(1)$ queries to block-encoding oracles for $H_1$ and $H_2$.
With the block-encoding oracle $\text{HAM-T}_{H_\mu}$, we can implement
\[\text{SEL}_0 = \sum_{k=0}^{N_p-1} \ket{k}\bra{k} \otimes  V_k^a(T), \]
a block-encoding of $\mathcal{U}(T)$, using the quantum singular value transformation (QSVT) \cite{gilyen2019quantum} for example, where $V_k^a(T)$ block-encodes $V_k(T)$ with
\begin{equation}\label{VaErr}
\|V_k^a(T) - V_k(T)\| \le \delta.
\end{equation}
This uses the oracles for $H_1$ and $H_2$
\begin{equation}\label{eq:times for H1}
\mathcal{O}\Big( (\alpha_1 \mu_{\max} + \alpha_2) T  + \log(1/\delta)\Big)
= \mathcal{O}(\alpha_H \mu_{\max} T + \log(1/\delta))
\end{equation}
times (see \cite[Corollary 16]{ACL2023LCH2}), where $\alpha_H \ge \alpha_i$, $i=1,2$.

Applying the block-encoding circuit to the initial input state $\ket{0}_{a'}\ket{\tilde{\bm{W}}_0}$ gives
\[\text{SEL}_0\ket{0}_{a'}\ket{\tilde{\bm{W}}_0} = \ket{0}_{a'}\mathcal{U}^a(T)\ket{\tilde{\bm{W}}_0} + \ket{\bot},\]
where $\mathcal{U}^a(T)$ is the approximation of $\mathcal{U}(T)$ and $\tilde{\bm{W}}_0 = (\Phi ^{-1}\otimes I) \bm{W}_h(0)$. This step only needs one query to the state preparation oracle $O_{\tilde{w}}$ for $\tilde{\bm{W}}_0$.

According to the preceding discussions, we may conclude that there exists a unitary $V_0$ such that
\[\ket{0^{n_a}} \ket{0^w} \quad \xrightarrow{ V_0 } \quad  \frac{1}{\eta_0} \ket{0^{n_a}} \otimes \tilde{\bb{W}}_h^{a} + \ket{\bot},\]
where $\tilde{\bb{W}}_h^{a}$ is the approximate solution of $\tilde{\bm{W}}_h$, given by
\[ \tilde{\bb{W}}_h^{a} = \mathcal{U}^a(T)\tilde{\bb{W}}_0 \quad \mbox{and} \quad \eta_0 = \|\tilde{\bb{W}}_0\| = \|\bb{W}_h(0)\|\le C_e \|\bm{u}_I\|\lesssim \frac{1}{\sqrt{\triangle p}}\sqrt{\|\bb{u}_0\|^2+T^2\|\bb{b}\|_{\mathrm{smax}}^2}.\]

\subsection{An abstract complexity analysis} \label{subsec:abstract}

In this section, we focus on the complexity analysis of the Schr\"odingerization.
According to \cite[Theorem 4.4]{JLM24SchrInhom}, the error between $\bm{u}_h=\e^{p_k}(\bra{k} \otimes \bra{0} \otimes I)\bm{W}_h$ and $\bm{u}$ consists of two parts: one arises from the truncation of the extended domain used for computation in \eqref{eq: L,R,criterion}, and the other results from the spectral discretization in $p$.   Suppose $L$ and $R$ are large enough satisfying \eqref{eq: L,R,criterion}, and $\triangle p \simeq \mu_{\max}^{-1}$ is small.

The original Schr\"odingerization method exhibits first-order convergence in $p$ due to the lack of regularity of $\psi(p) = \e^{-|p|}$ in the initial data in \eqref{u2v}. Consequently, achieving precision $\epsilon$ may require a small enough mesh size $\triangle p = \mathcal{O}(\epsilon)$. This results in the maximum absolute value among the discrete Fourier modes scaling as $\mathcal{O}(1/\epsilon)$, i.e., $\mu_{\max} = \mathcal{O}(1/\epsilon)$, which is not optimal because the query complexity linearly depends on $\mu_{\max}$ as shown in \eqref{eq:times for H1}.

The parameter $ \mu_{\max}$ is proportional to the inverse of the mesh size $\Delta p$, so it should be a function of $\epsilon$. To achieve better precision scaling, the natural idea is to adopt smoother extension of the warped phase transformation $\psi$ so the discrete Fourier transform~--~which is the spectral method~--~achieves high order (up to exponential)~--~accuracy.

Here, we derive the error estimate for the Fourier spectral discretization with smooth initializations, while the detailed proof is given in Section~\ref{sec:err}.
 Let $\psi(p) \in H^r((-L,R))$ with $\psi^{(k)}(p) \approx 0$ at $p = -L, R$ for $k\le r$. Then one can apply the discrete Fourier transform to $\psi(p)$. Denote its approximation by $P_h \psi$. Noting  that $\mu_{\max} =\frac{N_p \pi}{L+R}$, the standard approximation estimate \cite{Shenspectral} yields
 \begin{equation}\label{errPh}
 \| \psi - P_h \psi \|_{L^2((-L,R))} \lesssim \Big( \frac{R+L}{N_p} \Big)^r \|\psi^{(r)}\|_{L^2((-L,R))} = \big(\frac{\pi}{\mu_{\max}}\big)^{r}\|\psi^{(r)}\|_{L^2((-L,R))}.
 \end{equation}
If we assume that the right-hand side of \eqref{errPh} is of the same order as $\epsilon$, then
\[
 \mu_{\max} \simeq \pi (1/\epsilon)^{1/r} \|\psi^{(r)}\|_{L^2((-L,R))}^{1/r}.
\]

\begin{theorem}\label{thm:err w-wh}
Let $\bb{w}(t,p)$ be the exact solution to \eqref{u2v}, and let $\bm{W}_h(t)$ denote the solution of the discrete problem \eqref{heatww}.
Assume that $\psi \in H^r(\mathbb{R})$ and decays exponentially on $\mathbb{R}$.
Suppose the mesh size $\triangle p$ satisfies
\begin{equation}\label{mumax}
 (\triangle p)^{-1} \simeq \mu_{\max} \simeq \pi (1/\epsilon)^{1/r} \|\psi^{(r)}\|_{L^2((-L,R))}^{1/r},
\end{equation}
where $L$ and $R$ are chosen according to \eqref{eq: L,R,criterion}.
Then the following error estimate holds:
\begin{equation}\label{errw}
\| \bb{w}(T,p) - \bb{w}_h(T,p) \|_{L^2((-L,R))} \lesssim  \epsilon \|\bm{u}_{I}\|,
\end{equation}
where $\bb{w}_h$ is the continuous reconstruction of $\bm{W}_h$, given by
\begin{equation}\label{interpcoeff}
\bb{w}_h(t,p) = \sum_{l=0}^{N_p-1} \tilde{\bb{w}}_{l,h}(t)\,\phi_l(p),
\qquad
\tilde{\bb{w}}_{l,h}(t) = \frac{1}{N_p} \sum_{k=0}^{N_p-1} \big((\bra{k}\otimes I)\bm{W}_h\big) \,
\e^{ - \i \mu_l (p_k+L)}.
\end{equation}
\end{theorem}

For sufficiently large $r$, we have $(1/\epsilon)^{1/r} = \mathcal{O}(1)$. For example, we can assume $\e \le (1/\epsilon)^{1/r} \le \e^2$ and obtain
\begin{equation}\label{eq:mumax}
 \mu_{\max} \simeq  \|\psi^{(r)}\|_{L^2((-L,R))}^{1/r}
 \quad \text{for} \quad \frac12 \log \frac{1}{\epsilon} \le r \le \log \frac{1}{\epsilon}.
\end{equation}
Therefore, if we modify the original function $\psi(p)$ such that
\begin{equation}\label{boundpsi1}
\|\psi^{(r)}\|_{L^2((-L,R))}^{1/r} \lesssim \log(1/\epsilon) \qquad \text{when} \quad r \simeq \log(1/\epsilon),
\end{equation}
then substituting this bound into Eq.~\eqref{eq:times for H1} may imply that the non-unitary dynamic system \eqref{ODElinear} can be simulated from $t=0$ to $t=T$, within an error of $\epsilon$, achieving $\widetilde{\mathcal{O}}( \alpha_H T \log(1/\epsilon) )$ queries to the HAM-$\text{H}_{\mu}$ oracle.

Eq.~\eqref{boundpsi1} suggests that we should establish a linear growth of $ \|\psi^{(r)}\|_{L^2((-L,R))}^{1/r} $ with respect to $ r $. However, as will be shown later, a smooth extension of $ \psi $ alone is insufficient to achieve such growth and, therefore, cannot yield the optimal cost. This is because we can only derive $ \|\psi^{(r)}\|_{L^2((-L,R))}^{1/r} \le C r^{1/\beta} $ with $ \beta \in (0,1) $, which leads to a dependence on $ \epsilon $ of $ \mathcal{O}(\log^{1/\beta}(1/\epsilon)) $. To achieve the optimal convergence rate, one would need $\beta = 1$.

The following theorem establishes an abstract framework for the complexity analysis of the Schr\"odingerization method.

\begin{theorem}\label{thm:complexity}
Let $\varepsilon$ be a positive constant. Suppose that $L$ and $R$ satisfy the condition in \eqref{eq: L,R,criterion}.
In addition, we assume that the function $\psi\in H^r(\mathbb{R})$ in the initial data of \eqref{u2v}  decays exponentially on $\mathbb{R}$  and satisfies
\[\|\psi^{(r)}\|_{L^2((-L,R))}^{1/r} \le C r^{1/\beta}, \quad \beta \in (0,1], \]
where $r \simeq \log(1/\epsilon)$, $C$ is a constant independent of $\epsilon$,
and the inverse of the mesh size satisfies \eqref{mumax}. Here,
 $\epsilon$ satisfies
\[\epsilon \simeq \frac{\epsilon'}{(\log(1/\epsilon'))^{1/(2\beta)}}, \qquad \epsilon' =\frac{\varepsilon\|\bm{u}(T)\|}{\|\bm{u}_{I}\|}, \qquad
\|\bm{u}_{I}\| \simeq \|\bm{u}(0)\|+T\|\bm{b}\|_{\mathrm{smax}}, \]
 with $\|\bm{b}\|_{\mathrm{smax}}$ defined in \eqref{bmax}. This implies
\[\mu_{\max} \lesssim \Big(\log \frac{\|\bm{u}_{I}\|}{\varepsilon \|\bm{u}(T)\|}\Big)^{1/\beta}.\]
Then, there exists a quantum algorithm that prepares an $\varepsilon$-approximation of the state $\ket{\bm{u}(T)}$ with
	$\Omega(1)$ success probability and a flag indicating success, using
\[\mathcal{O}\Big( \frac{\|\bm{u}_{I}\| } {\|\bm{u}(T)\|}\alpha_H T \Big(\log \frac{\|\bm{u}_{I}\|}{\varepsilon \|\bm{u}(T)\|}\Big)^{1/\beta} \Big)\]
queries to the $\text{HAM-T}_{H_\mu}$ oracle, where $\alpha_H \geq \|H_i\|, i=1,2$, and using
	\begin{equation*}
	\mathcal{O}\Big(\frac{\|\bm{u}_{I}\| } {\|\bm{u}(T)\|}\Big)
	\end{equation*}
	queries to the state preparation oracle for $\tilde{\bm{w}}_0$.
\end{theorem}
\begin{proof}
Let $\bb{W}_h(T)$ and $\bb{W}_h^a(T)$ be the solutions associated with $\mathcal{U}$ and $\mathcal{U}^a$, respectively.
According to \eqref{eq:recovery}, one has
\[\bm{u}(T) = \e^{p_k}(\bra{k}\otimes \bra{0} \otimes I)\bm{W}(T), \qquad
\bm{u}_h^a(T) = \e^{p_k}(\bra{k}\otimes \bra{0} \otimes I)\bm{W}^a_h(T)\]
for some $k \in I_{\Diamond}$, where $\bm{W}(t) = \sum_{ki} w_i(t,p_k) \ket{k,i}$. Here, we can choose $p_k = \mathcal{O}(1)$.
Then we need to bound the error between $\ket{\bb{u}(T)}$ and $\ket{\bb{u}_h^a(T)}$.  Using the inequality $\| \frac{\bb{x}}{\|\bb{x}\|} - \frac{\bb{y}}{\|\bb{y}\|} \| \le 2 \frac{\|\bb{x} - \bb{y}\|}{\|\bb{x}\|}$ for two vectors $ \bb{x}, \bb{y} $, we obtain
\[\|\ket{\bm{u}(T)}-\ket{\bm{u}_h^a(T)}\| \le 2 \frac{\|\bm{u}(T) - \bm{u}_h^a(T)\|}{\|\bm{u}(T)\|}.\]
 This gives
\[\|\bm{u}(T) - \bm{u}_h^a(T)\| \le \e^{p_k} \| \bb{W}(T) - \bb{W}_h^a(T) \|, \qquad  k \in I_{\Diamond}.\]

By the triangle inequality,
\[\| \bb{W}(T) - \bb{W}_h^a(T) \| \le \| \bb{W}(T) - \bb{W}_h(T) \| + \| \bb{W}_h(T) - \bb{W}_h^a(T) \| = : \varepsilon_1 + \varepsilon_2.\]
For $\varepsilon_1$, one has
\begin{align*}
\varepsilon_1
& = \| \bb{W}(T) - \bb{W}_h(T) \| = (\sum_{k=0}^{N_p-1} \|\bb{w}(T,p_k) - \bb{w}_h(T,p_k)\|^2)^{1/2} \\
& \lesssim \frac{1}{\sqrt{\Delta p}} \| \bb{w}(T,p) - \bb{w}_h(T,p) \|_{L^2((-L,R))} \lesssim \mu_{\max}^{1/2} \epsilon \|\bm{u}_{I}\|,
\end{align*}
where we have used the estimate \eqref{errw}.
For $\varepsilon_2$, assuming $\|V_k^a(T) - V_k(T)\| \le \delta$, there holds
\begin{align*}
\varepsilon_2
 = \| \bb{W}_h(T) - \bb{W}_h^a(T) \|
\le \| \mathcal{U} - \mathcal{U}^a \|  \|\bb{W}_h(0)\|
\le \delta \|\bb{\psi}\|  \|\bb{u}_{I}\| \lesssim \mu_{\max}^{1/2} \delta \|\bb{u}_{I}\|.
\end{align*}
Therefore,
\[\|\ket{\bm{u}(T)}-\ket{\bm{u}_h^a(T)}\| \lesssim \mu_{\max}^{1/2} (\epsilon + \delta) \frac{\|\bm{u}_{I}\|}{\|\bm{u}(T)\|}.\]

The condition on $\psi$ implies that $\mu_{\max} \lesssim \log^{1/\beta}(1/\epsilon)$.  Given the above equations, we can require
\[\Big(\log\frac{1}{\epsilon}\Big)^{1/(2\beta)} \epsilon \frac{\|\bm{u}_{I}\|}{\|\bm{u}(T)\|} \simeq \frac{\varepsilon}{2}, \qquad
\Big(\log\frac{1}{\epsilon}\Big)^{1/(2\beta)} \delta  \frac{ \|\bb{u}_{I}\|}{\|\bm{u}(T)\|} \simeq \frac{\varepsilon}{2}.\]
From the first equation we can get
\[\epsilon \simeq \frac{\varepsilon'}{(\log(1/\epsilon'))^{1/(2\beta)}}, \qquad \epsilon' =\frac{\varepsilon\|\bm{u}(T)\|}{\|\bm{u}_{I}\|}. \]
This yields
\[\mu_{\max}
 \lesssim \Big(\log \frac{(\log(1/\epsilon'))^{1/(2\beta)}} {\varepsilon'} \Big)^{1/\beta}
\simeq \Big(\log \frac{\|\bm{u}_{I}\|}{\varepsilon \|\bm{u}(T)\|}\Big)^{1/\beta}, \]
\[\frac{1}{\delta} \simeq \Big(\log \frac{\|\bm{u}_{I}\|}{\varepsilon \|\bm{u}(T)\|}\Big)^{1/\beta} \frac{ \|\bm{u}_{I}\| }{\varepsilon \|\bm{u}(T)\|}.\]
Plugging the above quantities into \eqref{eq:times for H1}, we obtain
\begin{align*}
\alpha_H T \mu_{\max}+ \log \frac{1}{\delta}
=\mathcal{O}\Big( \alpha_H T \Big(\log \frac{\|\bm{u}_{I}\|}{\varepsilon \|\bm{u}(T)\|}\Big)^{1/\beta} \Big) .
\end{align*}
The proof is finished by multiplying the repeated times shown in \eqref{gtimes}.
\end{proof}

\begin{remark}
From this theorem, we observe that the optimal complexity is achieved when $\beta = 1$. However, we should relax the assumption that $\psi(p) = \e^{-p}$ for $p > 0$, and instead pursue an approximation, as discussed later. Therefore, $\bm{u}_f$ is not exactly $\e^p \bm{w}$, but rather an approximation of $\bm{u}_f$. In Theorem \ref{thm:recovery2}, we provide an error estimate related to this approximation. Taking this error into account still does not affect our final complexity analysis.
\end{remark}

\section{Near-optimal dependence with smooth initializations} \label{sec:smoothmethod}

In this section, we demonstrate that the smooth extension to $p < 0$ for $\e^{-p}$ is sufficient to ensure near-optimal dependence on matrix queries. However, as mentioned earlier, the smooth extension alone does not guarantee optimal precision dependence.  We present three methods for constructing such a smooth extension.

\subsection{The cut-off function}\label{subsec:cut-off func}

\begin{figure}[!htb]
  \centering
  \includegraphics[scale=0.5]{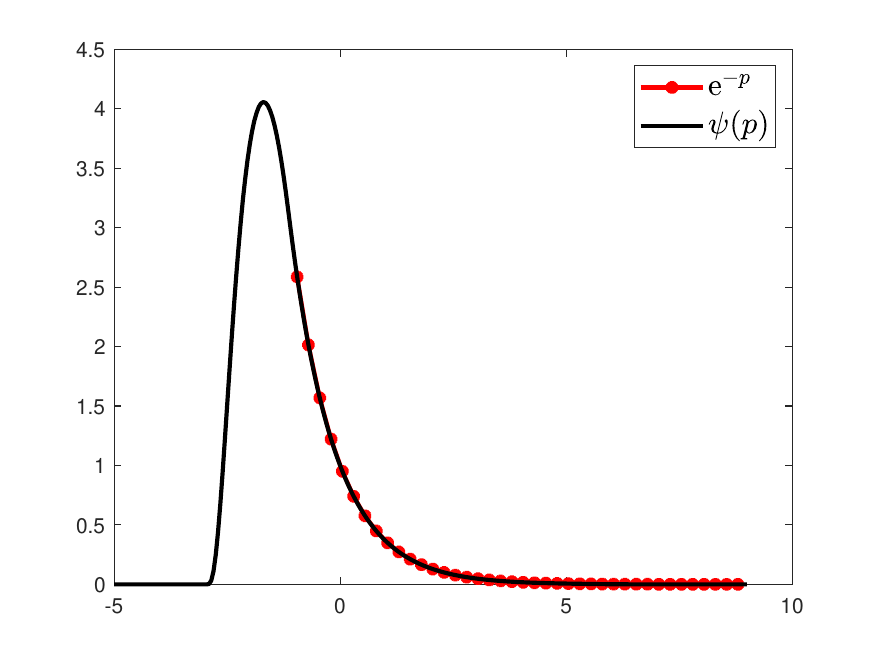}\\
  \caption{A smooth extension of $\e^{-p}$} \label{fig:smooth}
\end{figure}

As illustrated in Fig.~\ref{fig:smooth}, we seek a smooth function $\psi(p)$ such that $\psi(p) = \e^{-p}$ for $p > 0$, and it decays rapidly on the negative axis. This problem can be reduced to finding a smooth approximation $\zeta(p)$ of the step function
\begin{equation}\label{stepfunc}
H(p) = \begin{cases}
0, & p < 0, \\
1, & p \geq 0,
\end{cases}
\end{equation}
and then defining $\psi(p) = \zeta(p) \e^{-p}$.
Given that the discrete Fourier transform requires periodic boundary conditions and recognizing the exponential decay of $\e^{-p}$ for positive $p$, we further require $\zeta(p)$ to have compact support. This leads us to consider a cut-off function, which ensures the smoothness of the step function and provides compact support via convolution.
The convolution operation is a widely used method for smoothing the step function. It generates various smooth approximations by selecting an appropriate convolution kernel. The cut-off function is a prime example of this approach and has become a fundamental tool in the analysis of partial differential equations. In this subsection, we apply the cut-off technique to construct the desired smooth extension.

We begin by recalling the mollifier, defined as
\begin{equation}\label{mollifier}
\eta(p) = \begin{cases}
  \frac{1}{C}\exp \left( \frac{1}{| p |^2 - 1} \right), & \quad | p | < 1  \\
  0, & \quad | p | \ge 1
\end{cases}, \qquad C = \int_{B_1} \exp\left(\frac{1}{|p|^2 - 1}\right) \mathrm{d}p,
\end{equation}
where $ B_1 $ denotes the unit ball in $\mathbb{R}^n$ and $ C $ is the normalization constant ensuring $ \int_{\mathbb{R}^n} \eta(p) \, \d p = 1 $. This function belongs to $ C_0^\infty(\mathbb{R}^n) $ with support $ \overline{B_1} $.

\begin{lemma}\label{lem:etak}
The mollifier satisfies the following estimate for its derivatives in one dimension:
\begin{equation}\label{mollifierbound}
|\eta^{(k)} (p) | \lesssim C(k): = 20^k k!\e^{-2k} (2k)^{2k},\quad \forall p\in \mathbb{R}.
\end{equation}
\end{lemma}

\begin{proof}
We only need to consider $\theta(p) = C\eta(p)$. A direct calculation gives
\[\theta^{(k)}(p) = Q_k(p)(1 - p^2)^{-2k}\exp\left(\frac{1}{|p|^2 - 1}\right), \qquad |p| < 1.\]
where $Q_k$ is a polynomial and can be recursively defined by
\[Q_0(p) = 1,\quad Q_{k+1}(p) = (1 - p^2)^2 Q'_k(p) + 2p(2k - 1 - 2kp^2) Q_k(p).\]
Let $t = \frac{1}{1 - p^2} \in (1, \infty)$.
For $k \ge 1$, one has
\[ \theta^{(k)}(p) = Q_k(p) \e^{-t} t^{2k}, \quad |p| < 1, \quad 1 < t < \infty.\]
Since $\e^{-t} t^{2k}$ achieves the maximum value at $t = 2k$ when $t>1$, we have
\[ |\theta^{(k)}(p)| \le |Q_k(p)|  \e^{-2k} (2k)^{2k}.\]
It is simple to find that $Q_k(p)$ is a polynomial of order $3k$. Let
\[Q_k(p) = \sum_{j=0}^{3k} a_{k,j}p^j, \qquad |p|<1,\]
and we define the maximum coefficient in absolute value as
\[A_k = \max_{0\le j \le 3k} |a_{k,j}|, \qquad k = 1,2,\cdots\]
Then it holds
\[|Q_k(p)| \le (3k+1) A_k, \qquad k = 1,2, \cdots\]

Through careful calculation, it is found that
\begin{align*}
(1-p^2)^2 Q_k'(p)
& = (1-2p^2 + p^4)\sum_{j=0}^{3k-1} (j+1) a_{k,j+1}p^j
%& = \sum_{j=0}^{2k-2} (j+1) a_{k,j+1}x^j - \sum_{j=0}^{2k-2} 2(j+1) a_{k,j+1}x^{j+2} + \sum_{j=0}^{2k-2} (j+1) a_{k,j+1}x^{j+4} \\
%& = \sum_{j=0}^{2k-2} (j+1) a_{k,j+1}x^j - \sum_{j=2}^{2k} 2(j-1) a_{k,j-1}x^j + \sum_{j=0}^{2k+2} (j-3) a_{k,j-3}x^j \\
%& = \sum_{j=0}^{3k-1} (j+1) a_{k,j+1}p^j - \sum_{j=2}^{3k} 2(j-1) a_{k,j-1}p^j + \sum_{j=0}^{2k+2} (j-3) a_{k,j-3}p^j\\
 =: \sum_{m=0}^{3k+3} S(m) p^m,
\end{align*}
where $S(m)$ is obtained by collecting powers of $p^m$. Similarly,
\[
2p(2k-1 - 2kp^2) Q_k(p) =: \sum_{m=0}^{3k+3} T(m) p^m.
\]

By examining the recursive formulas, it is clear that each coefficient satisfies
\[
|S(m)| \lesssim 12 k A_k, \quad |T(m)| \lesssim 8 k A_k, \quad m \le 3k+3.
\]
Hence
\[
A_{k+1} \le 20 k A_k, \quad k = 1,2,\cdots
\]
with $A_1 = 2$. By induction, we have
\[
A_k \le 2 \cdot 20^{k-1} (k-1)!.
\]
Therefore,
\[
|Q_k(p)| \le (3k+1) A_k \le 20^k k!,
\]
up to a universal multiplicative constant. Combining with the previous estimate yields \eqref{mollifierbound}.

% where
% \[S(m) = \begin{cases}
% a_{k,1},   \qquad & m = 0,\\
% 2a_{k,2},  \qquad & m = 1,\\
% 3a_{k,3}-2a_{k,1},  \qquad & m = 2,\\
% 4a_{k,4}-4a_{k,2},  \qquad & m = 3,\\
% (m+1)a_{k,m+1}-2(m-1)a_{k,m-1} + (m-3)a_{k,m-3},  \qquad & 4\le m \le 2k-2,\\
% 2ka_{k,2k}-2(2k-2)a_{k,2k-2} + (2k-4) a_{k,2k-4},  \qquad & m = 2k-1,\\
% -2(2k-1)a_{k,2k-1}+(2k-3)a_{k,2k-3},  \qquad & m = 2k,\\
% (2k-2)a_{k,2k-2},  \qquad & m = 2k+1,\\
% (2k-1)a_{k,2k-1},  \qquad & m = 2k+2,\\
% \end{cases}\]
% which implies that $S(m)$ is bounded by $8k A_k$. The same process gives
% \[
% 2p(2k - 1 - 2kp^2) Q_k(p) = (4k-2) \sum_{m=1}^{2k} a_{k,m-1} p^m - 4k \sum_{m=3}^{2k+2} a_{k,m-3} p^m =:\sum_{m=0}^{2k+2} T(m) p^m,
% \]
% where
% \[
% T(m) =
% \begin{cases}
% 0, & m=0, \\
% (4k-2) a_{k,0}, & m=1, \\
% (4k-2) a_{k,1}, & m=2, \\
% (4k-2) a_{k,m-1} - 4k a_{k,m-3}, & 3 \leq m \leq 2k, \\
% -4k a_{k,2k-2}, & m=2k+1, \\
% -4k a_{k,2k-1}, & m=2k+2, \\
% \end{cases},
% \]
% which is also bounded by $8k A_k$. Therefore, we have
% \[A_{k+1} \le 16 k A_k, \qquad  k = 1,2,\cdots\]
% with $A_1 = 2$.
% This gives
% \[
% A_k \leq 2 \prod_{j=1}^{k-1} 16j = 2 \cdot 16^{k-1} (k-1)!,
% \]
% and thus,
% \[
% |Q_k(p)| \le 2k \cdot 2 \cdot 16^{k-1} (k-1)! = \frac{k}{4} \cdot 16^k (k-1)!.
% \]
% This completes the proof.
\end{proof}

For any $ \varepsilon > 0 $, we can rescale the function such that its support becomes $ \overline{B_\varepsilon} $, a closed ball of radius $ \varepsilon $. The rescaled function is given by $\eta_\varepsilon(p) = \frac{1}{\varepsilon^n} \eta \left( \frac{p}{\varepsilon} \right)$.
For a function $ u \in L_{\text{loc}}^1(\Omega) $, the mollifier operator $ J_\varepsilon $ is defined through convolution as
\begin{equation}\label{eq:mollifier1}
 J_\varepsilon u(p) = (\eta_\varepsilon * u)(p) = \int_{\Omega} \eta_\varepsilon(p - y) u(y) \d y
= \int_{B_\varepsilon(p)} \eta_\varepsilon(p - y) u(y) \d y, \quad p \in \Omega_\varepsilon,
\end{equation}
where the domain $ \Omega_\varepsilon $ is defined by
\[ \Omega_\varepsilon = \left\{ p \in \Omega : \overline{B_\varepsilon(p)} \subset \Omega \right\} = \left\{ p \in \Omega : \text{dist}(p,\partial \Omega) > \varepsilon \right\}. \]
It can be verified that $ J_\varepsilon u \in C^{\infty}(\Omega_\varepsilon) $ for every $ u \in L_{\text{loc}}^1(\Omega) $. Furthermore, if $ \text{supp} \{ u \} \Subset \Omega $, denoting $ \delta = \text{dist}(\text{supp} \{ u \}, \partial \Omega) $, then for $ \varepsilon < \delta/4 $, we have $ J_\varepsilon u \in C_0^\infty (\Omega) $ with $ \text{supp} \{ J_\varepsilon u \} \subset \Omega_\varepsilon $.

\begin{figure}[!htb]
  \centering
  \includegraphics[scale=0.5]{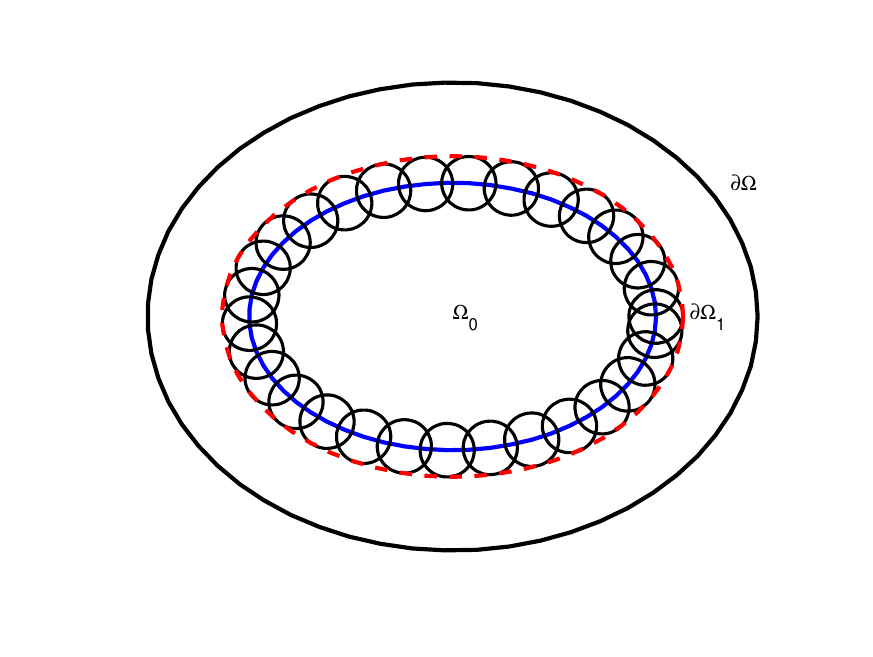}\\[-0.9cm]
  \caption{A snapshot of the domain of the cut-off function in $\mathbb{R}^2$} \label{fig:cutoff}
\end{figure}

Now we are ready to describe the construction of the cut-off function with the domain illustrated in Fig.~\ref{fig:cutoff}.

\begin{lemma}[cut-off function]
Let $\Omega \subset \mathbb{R}^n$ be a non-empty open set, and ${\Omega _0} \Subset \Omega $. Define
\[
\delta = \text{dist}(\Omega _0, \partial \Omega), \quad d = \frac{\delta}{4}, \quad \Omega_1 = \{ p \in \Omega : \text{dist}(p, \Omega _0) < d \}.
\]
Let $\phi(p) = \chi _{\Omega _1}(p)$ denote the indicator function of $\Omega_1$. Then $\zeta = J_d{\phi}$ satisfies
\[
\begin{cases}
 \zeta \in C_0^\infty (\Omega ), & \text{supp} \{ \zeta \} \subset K_d, \\
\zeta (p) \equiv 1, & p \in \Omega _0, \\
0 \le \zeta (p) \le 1, & p \in \Omega ,
\end{cases}
\]
where
\[
{K_d} = \{ p \in \Omega : \text{dist}(p, \Omega _1) \le d \} = \{ p \in \Omega : \text{dist}(p, \Omega _0) \le 2d \}.
\]
The function $\zeta$ is referred to as the cut-off function relative to the subset $\Omega _0$ in $\Omega$.
\end{lemma}

The one-dimensional cut-off function satisfies the following estimate for its derivatives \cite{gilbarg2001elliptic,evans2010partial}:
\begin{equation}\label{zetak}
| \zeta^{(k)}(p) | \lesssim \frac{C(k)}{d^k},
\end{equation}
where $C(k)$ is defined in \eqref{mollifierbound}.

For the Schr\"odingerization method, we set $\Omega_0 = (-1,R)$ and $d \ge 1$. Let
\begin{equation}\label{cutoffmethod}
    \psi(p) = \zeta(p) \e^{-p}.
\end{equation}
Then it holds that $\text{supp}\{\psi\} \subset [-(1+2d), R+2d]$. The cut-off function and the resulting smooth extension are shown in Fig.~\ref{fig:cutoff1} for $R=5$ and $d=1$.

\begin{theorem}
For any $\epsilon>0$, let $d = r \simeq \log(1/\epsilon)$. Then the smooth initialization function $\psi(p)$ defined in \eqref{cutoffmethod} satisfies
\[\|\psi^{(r)}\|_{L^2(\mathbb{R})}^{1/r} \lesssim  r^2 \simeq \log^2 (1/\epsilon). \]
\end{theorem}
\begin{proof}
Let $\xi^{(k)}(p) = \zeta^{(k)}(p) \e^{-p}$. Noting that
\[\psi^{(r)}(p) = \sum_{k=0}^r C_r^k \zeta^{(k)}(p) (\e^{-p})^{(r-k)} = \sum_{k=0}^r (-1)^{r-k} C_r^k \zeta^{(k)}(p) \e^{-p},\]
by careful calculation, for any $p\in\mathbb{R}$, one gets
\begin{align}
|\psi^{(r)}(p)|
& \le ((C_r^0)^2 + \cdots +(C_r^r)^2)^{1/2} ( |\xi^{(0)}|^2 + \cdots + |\xi^{(r)}|^2 )^{1/2} \nonumber\\
& = (C_{2r}^r)^{1/2} ( |\xi^{(0)}|^2 + \cdots + |\xi^{(r)}|^2 )^{1/2} \le 2^r ( |\xi^{(0)}|^2 + \cdots + |\xi^{(r)}|^2 )^{1/2}, \label{psirbound}
\end{align}
where the combinatorial equality can be found in \cite{Stanley2011Enumerative}. In addition, we have used the Cauchy-Schwarz inequality and the fact that $C_{2r}^r \le \sum_{k=0}^{2r} C_{2r}^k = (1+1)^{2r} = 4^r$. According to \eqref{mollifierbound} and \eqref{cutoffmethod}, if we take $d = r$, then there holds
\begin{align*}
|\psi^{(r)}(p)|
& \le 2^r r^{1/2} \max_{0\le k\leq r} |\zeta^{(k)}(p) \e^{-p}| \lesssim 2^r r^{1/2}\max_{0\le k\leq r} \frac{C(k)}{d^k} \e^{-p}\\
& =  2^r r^{1/2} \max_{0\le k\leq r} \frac{20^k k!\e^{-2k} (2k)^{2k}}{d^k} \e^{-p} \le 40^r r^{1/2}(2r)^{2r}.
\end{align*}
The above equation together with $\text{supp}\{\psi\} \subset [-(1+2d), R+2d]$ and $d=r$ yields
\begin{align*}
\int_{\mathbb{R}} |\psi^{(r)}(p)|^2 \d p
\lesssim (40^r r^{1/2}(2r)^{2r})^2 \int_{-(1+2d)}^{\infty} \e^{-2p} \d p \lesssim (40^r r^{1/2}(2r)^{2r})^2 \e^{2r}.
\end{align*}
Therefore, by taking $r \simeq \log (1/\epsilon)$, we obtain
\[\|\psi^{(r)}\|_{L^2(\mathbb{R})}^{1/r} \lesssim (40^r r^{1/2} (2r)^{2r} )^{1/r} \lesssim r^2\lesssim \log^2 (1/\epsilon). \]
This completes the proof.
\end{proof}

According to Theorem \ref{thm:complexity}, the above result implies a sub-optimal precision dependence.
Ref.~\cite{Johnson15bump} states that the mollifier $ \eta(p) $ decays in the Fourier domain asymptotically as $ \hat{\eta}(w) = (|w|^{-3/4}) \e^{-\sqrt{|w|}} $. This exhibits super-polynomial decay, as the exponent involves a square root of $ |w| $, rather than exponential decay. This corresponds to the case $ \beta = 1/2 $ in \cite{ACL2023LCH2}, which consequently leads to $ \mathcal{O}(\log^2 (1/\epsilon)) $ for the LCHS method in terms of precision.

\begin{figure}[!htb]
  \centering
  \includegraphics[scale=0.5]{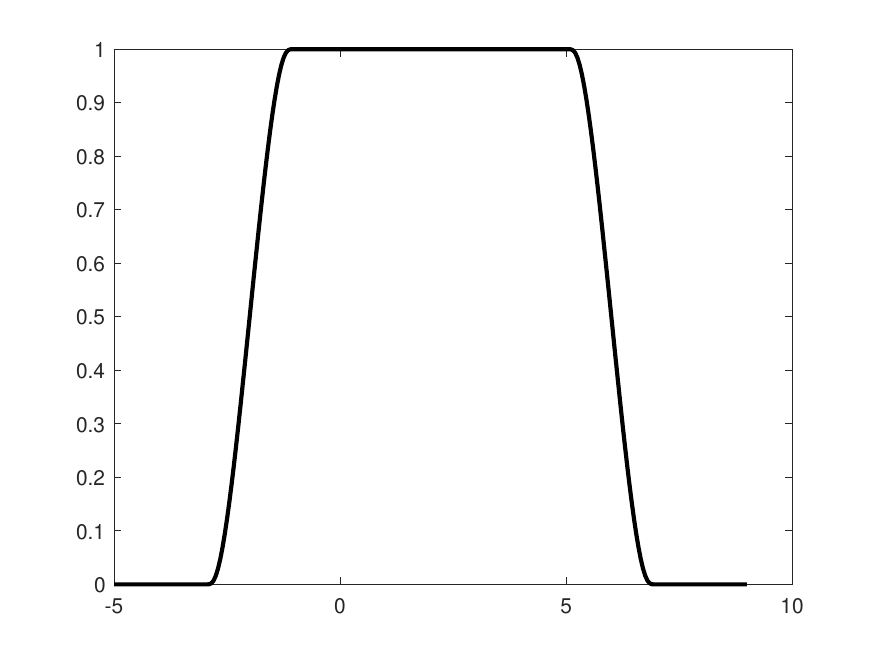}
  \includegraphics[scale=0.5]{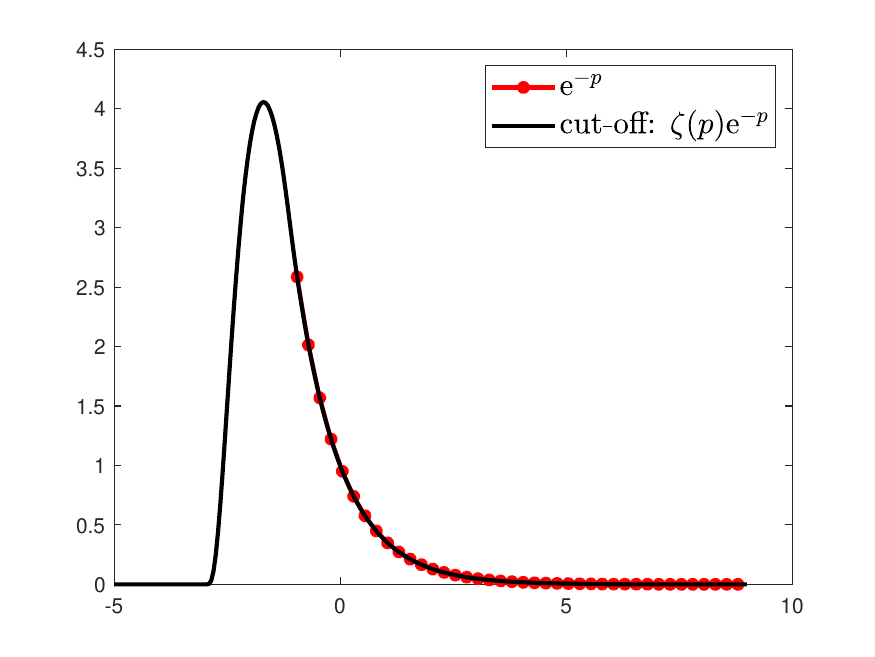}\\
  \caption{The cut-off function and the resulting smooth extension} \label{fig:cutoff1}
\end{figure}

\subsection{Higher-order interpolation}

An alternative approach to constructing the smooth function in the extended domain is to utilize high-order interpolations.
To do so, we rewrite the initial function  as
\[
\psi(p) = \begin{cases}
h(p), &  p \in (-\infty, 0], \\
\e^{-p}, & p \in (0, \infty),
\end{cases}
\]
where $h(p)$ is defined by
\begin{equation}
	h(p) = \Cp_{2r-1} (p),\quad p\in [-1,0],   \quad \quad   h(p) = \e^{p},\quad p\in (-\infty,-1).
\end{equation}
Here, $\Cp_{2r-1} (p)$ is a Hermite interpolation polynomial whose degree does not exceed  $(2r-1)$ \cite[Section~2.1.5]{SBBGW80},  satisfying
\begin{equation*}
	\begin{aligned}
		\big(\partial_{p^{k}}\Cp_{2r-1}(p)\big)|_{p=0}&=\big(\partial_{p^{k}}(\e^{-p})\big)|_{p=0} =(-1)^{k},\\
		\big(\partial_{p^{k}}\Cp_{2r-1}(p)\big)|_{p=-1}&=\big(\partial_{p^{k}}(\e^{p})\big)|_{p=-1}=\e^{-1},
	\end{aligned}
\end{equation*}
where $0\leq k\leq r-1$ is an integer.
It is simple to check that $\psi\in C^{r-1}(\bbR)$ and $\psi(p) \in H^{r}((-L,R))$ after restricting the extended domain to a limited interval.
The explicit formula of $\Cp_{2r-1}(p)$ is given by
\begin{equation*}
	\Cp_{2r-1}(p) = \sum_{k=0}^{r-1}\big(\e^{-1}L_{0k}(p) + (-1)^k L_{1k}(p)\big),
\end{equation*}
where $L_{ik}$ are generalized Lagrange polynomials defined recursively for $k=r-2$, $r-3$, $\cdots$, 0,
\begin{equation*}
	L_{0k}(p) := l_{0k}(p) - \sum_{\nu =k+1}^{r-1} l_{0k}^{(\nu)}(-1) L_{0\nu}(p),\quad
	L_{1k}(p) := l_{1k}(p) - \sum_{\nu =k+1}^{r-1} l_{1k}^{(\nu)}(0) L_{1\nu}(p),
	\end{equation*}
 with the starting polynomial for $k=r-1$
 \begin{equation*}
     L_{0r-1}(p) := l_{0r-1}(p),\qquad
     L_{1r-1}(p) := l_{1r-1}(p).
 \end{equation*}
The auxiliary polynomials  are
\begin{equation*}
	l_{0k}(p):= \frac{(-1)^r (p+1)^{k}p^r}{k!},\qquad
	l_{1k}(p):= \frac{p^k(p+1)^r}{k!}.
\end{equation*}

According to the discussion in \cite{JLM24SchrBackward}, the target variable $\bb{u}(t)=\e^p \bb{v}(t,p)$ still holds for all $p>0$, since we do not care the solution when $p<0$.  We provide the snapshots of $\psi$ for $r = 2,4,6,8,10$ in Fig.~\ref{fig:high_order inter}.
\begin{figure}
\centering
	\includegraphics[width=0.5\linewidth]{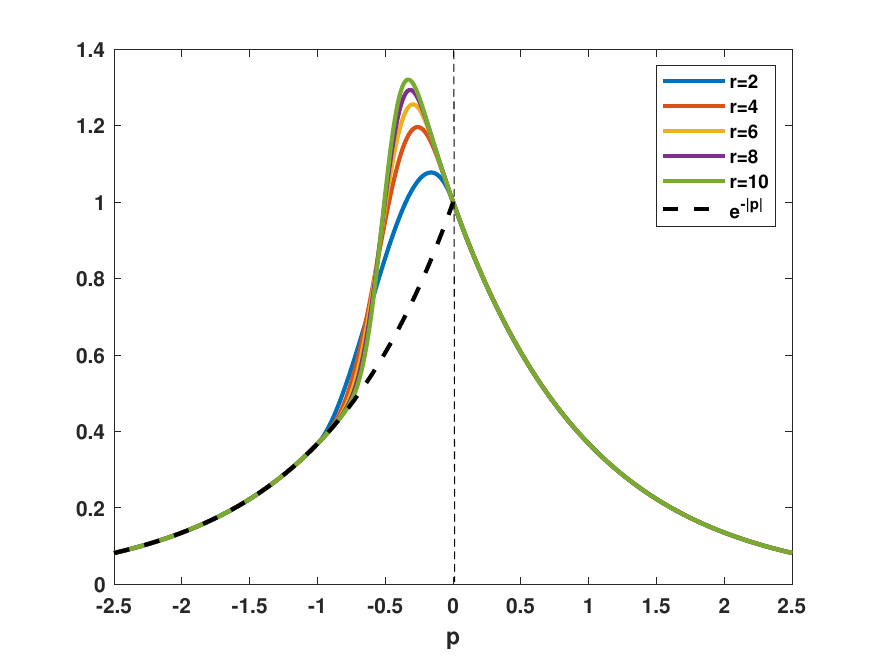}
\caption{The smooth initial data of $\psi(p)$ by using high-order interpolation.}
\label{fig:high_order inter}
\end{figure}

By employing the mollifier technique described in Section~\ref{subsec:cut-off func}, we can identify a smooth function $\varphi \in C^{\infty}(\mathbb{R}) $ such that $\varphi(p) \big|_{(-\infty,-1) \cup (0,\infty)} = \e^{-|p|}$. Consequently, $\Cp_{2r-1}(p)$ can be interpreted as an interpolation of $\varphi$.
Given that $\|\varphi^{(r)}\|_{L^{\infty}(-1,0)}$ is bounded, the $H^r$ norm of $\psi$ over the finite interval $[-L,R]$ remains bounded.
Unfortunately, while $\psi$ is contained in $H^r$ with any fixed $r\geq 1$, its exclusion from the $C^{\infty}$-class necessarily restricts the achievable complexity to a quasi-optimal order.
For the Schr\"odingerization method, the truncation of the extended domain must satisfy \eqref{eq: L,R,criterion}, it follows that $\bm{w}(t,-L)\approx \bm{w}(t,R)\approx 0$ \cite{JLM24SchrInhom}. Therefore, spectral methods can be used.

\subsection{Fourier transform} \label{subsec:FourierSmooth}

Building upon similar principles and utilizing the continuous Fourier transform in $ p $, \cite{ALL2023LCH} introduces an algorithm for implementing Linear Combination of Hamiltonian Simulation (LCHS).
The original LCHS approach in \cite{ALL2023LCH} is also a first-order method due to the slow decay rate of the integrand as a function of $ k $, where $ k $ is the continuous Fourier mode. In the continuous scenario, the integrand function with respect to $ k $ is the Fourier transform of $ \psi(p) = \e^{-|p|} $, given by $ \frac{1}{\pi(1+k^2)} $. It decays only quadratically, necessitating the truncation interval choice of $ [-K,K] $ with $ K = \log(1/\varepsilon) $. This introduces a computational overhead of $ \mathcal{O}(1/\varepsilon) $, as $ k H_1(s) - H_2(s) $ may have a spectral norm as large as $ K \|H_1(s)\| $.
This limitation was addressed in \cite{ACL2023LCH2} by replacing the original integrand with a new kernel function decaying at a near-exponential rate $ \e^{-c |k|^\beta} $, where $ \beta \in (0,1) $. Consequently, they no longer need to truncate the interval at $ K = \mathcal{O}(1/\varepsilon) $, and instead use the much smaller cutoff $ K = (\log(1/\varepsilon))^{1/\beta} $. The improved LCHS method requires
\[ \widetilde{\mathcal{O}}\left( \frac{\|\bb{u}(0)\| + \|\bb{b}\|_{L^1}}{\|\bb{u}(T)\|} \alpha_H T (\log(1/\varepsilon))^{\gamma} \right) \]
queries to the HAM-T oracle, where $\|\bm{b}\|_{L^1}=\int_0^T \|\bm{b}(s)\|\d s$, and
$\gamma = 1 + 1/\beta$ and $\beta$ for linear systems with time-dependent and time-independent coefficients, respectively. This leads to an exponential reduction in the Hamiltonian simulation time with respect to $ \varepsilon $ compared to the original LCHS. Since $ \beta \in (0,1) $, it holds that $ 1/\beta > 1 $, indicating sub-optimal behavior with respect to queries to the HAM-T oracle.

\begin{figure}[!htb]
  \centering
  \includegraphics[scale=0.25]{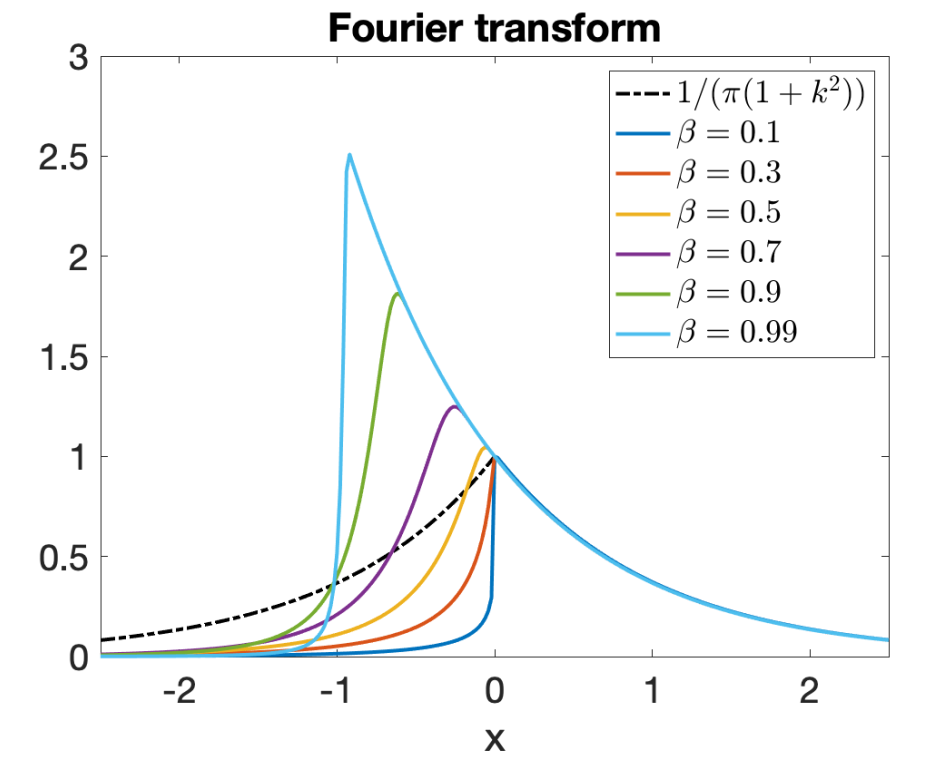}\\
  \caption{Fourier transforms of the kernel functions (see Figure 3 of \cite{ACL2023LCH2})} \label{fig:FourierSmooth}
\end{figure}

We emphasize that the fundamental principle remains consistent with our approach with smooth initialization. Indeed, as illustrated in Fig.~\ref{fig:FourierSmooth}, the Fourier transforms of their kernel functions are $\e^{-x}$ for $x > 0$, whereas they exhibit significant differences on the negative real axis. Therefore, the transformed kernel function can be interpreted as a smooth extension in the $p$ space.

\section{Establishing the optimal precision dependence}\label{sec:optimal scaling}

In this section, we demonstrate how to achieve optimal dependence on matrix queries by choosing suitable smooth initializations.

\subsection{Conditions for optimal precision dependence}\label{subsec:conditions}

According to the discussion in Section \ref{subsec:abstract}, if we require the original function $\psi(p)$ to satisfy
\begin{equation}\label{boundpsi}
\|\psi^{(r)}\|_{L^2((-L,R))}^{1/r} \lesssim \log(1/\epsilon) \qquad \text{when} \quad r \simeq \log(1/\epsilon),
\end{equation}
then substituting this bound into Theorem \ref{thm:complexity} implies that the non-unitary dynamic system \eqref{ODElinear} can be simulated from $t=0$ to $t=T$, within an error of $\varepsilon$, achieving $\widetilde{\mathcal{O}}( \alpha_H T \log(1/\epsilon) )$ queries to the HAM-$\text{H}_{\mu}$ oracle. This achieves optimal dependence on matrix queries.

However, it is impossible to achieve the bound in \eqref{boundpsi} while still requiring that $\psi(p) = \e^{-p}$ for $p \ge 0$. The reason is as follows. Let $\hat{\psi}(w)$ denote the inverse Fourier transform of $\psi(p)$, defined by
\[\hat{\psi}(w) = \frac{1}{\sqrt{2\pi}} \int_{-\infty}^{\infty} \psi(p) \e^{\i w p} \d p.\]
It follows from the Parseval's identity that
\begin{equation}\label{Parseva}
\|\psi^{(r)}\|_{L^2(\mathbb{R})} =  \|w^r\hat{\psi} \|_{L^2(\mathbb{R})}.
\end{equation}
\begin{itemize}
  \item If $\psi(p)$ is $C^\infty$ but not analytic, one can only expect super-polynomial decay of $\hat{\psi}$:
  \[|\hat{\psi}(w)| \le C \e^{-c |w|^{\beta}}, \qquad 0<\beta<1.\]
 For example, the mollifier $\eta(p)$ in \eqref{mollifier} decays in the Fourier domain asymptotically as $ \hat{\eta}(w) = (|w|^{-3/4}) \e^{-\sqrt{|w|}}$ \cite{Johnson15bump}.
  One can show that in this case,
  \[ \|w^r\hat{\psi} \|_{L^2(\mathbb{R})} \le C r^{r/\beta},\]
  which implies
  \[\|\psi^{(r)}\|_{L^2(\mathbb{R})}^{1/r} \le C  r^{1/\beta} = \mathcal{O}(\log^{1/\beta}(1/\epsilon))\]
  when $r \simeq \log(1/\epsilon)$.

  \item If $\psi(p)$ is analytic, according to the Paley-Wiener theorem, we have the exponential decay  of $\hat{\psi}$:
  \[|\hat{\psi}(w)| \le C \e^{-c |w|},\]
  which implies the desired bound
  \[\|\psi^{(r)}\|_{L^2(\mathbb{R})}^{1/r} =\mathcal{O}(r)= \mathcal{O}(\log(1/\epsilon))\]
  when $r \simeq \log(1/\epsilon)$.
\end{itemize}

As demonstrated in \cite{ACL2023LCH2}, the LCHS method cannot achieve optimal scaling with respect to the error tolerance $\varepsilon$. The underlying reason is that $\hat{\psi}(w)$ cannot exhibit exponential decay under the requirement that $\psi(p) = \e^{-p}$ for $p \ge 0$. The Schr\"odingerization and LCHS frameworks are closely related, as they share similar foundational mathematical principles. This connection further implies that the Schr\"odingerization method also attains only suboptimal scaling in $\epsilon$.

Based on the above discussion, to achieve optimal dependence on precision, we should relax the condition that $\psi(p) = \e^{-p}$ for $p \ge 0$ and instead seek an approximate version.
The required assumptions on $\psi(p)$ are as follows:
\begin{algorithm}[H]
\caption*{\textbf{Conditions on $\psi(p)$}}
\begin{itemize}
  \item[(\textbf{H1})] $\psi(p)$ exhibits exponential decay on $\mathbb{R}$ such that
   \[\psi(p) \lesssim \e^{-|p|} \leq 2\epsilon,\qquad \; p\in (-\infty, -L+\lambda_{\max}^-(H_1)T)\cup (R-\lambda_{\max}^+(H_1)T,+\infty),\]
  where $L$ and $R$ satisfy \eqref{eq: L,R,criterion}.
  \item[(\textbf{H2})] For $p \in [p_*,R]$, the condition $|\psi(p) - \e^{-p}| \le \epsilon$ holds, where $p_*\leq 1/2$,

  \item[(\textbf{H3})] There exist a constant $C$, independent of $r$ and $\epsilon$, such that
  \begin{equation}\label{psir}
  \|\psi^{(r)}\|_{L^2(\mathbb{R})}^{1/r} \le C  r, \qquad r \simeq \log(1/\epsilon).
  \end{equation}
\end{itemize}
\end{algorithm}

   The first assumption \textbf{(H1)} on $\psi$ is required to approximate the infinite domain problem by a periodic problem on the truncated interval $[-L,R]$, with specific error estimates provided in Theorem \ref{thm:truncation err}. The second assumption ensures the recovery of the target variable from the warped phase transformation in the Schr\"odinger-type formulation. It is noteworthy that if $\psi(p) = \e^{-p}$, the target variable can be recovered without error (see Theorem \ref{thm:recovery1}). If $\psi(p)$ satisfies an approximate condition \textbf{(H2)} , the corresponding error estimates are given in Theorem \ref{thm:recovery2}. The third assumption \textbf{(H3)} is introduced to derive the optimal estimate for $\mu_{\max}$ , with the main result presented in Theorem \ref{thm:complexity}.

%\begin{remark}
%The solution can be recovered within $(p_*+\lambda_{\max}^+(H_1)T,~~R-\lambda_{\max}^-(H_1)T)$.
%{\color{violet} It is worth noting that in practical computations we impose periodic boundary conditions
%in a finite domain $[-L,R]$ instead of considering the whole space. As a result, at time $T$, the values of $\bm{w}(T,p)$ for $p>R-\lambda_{\max}^-(H_1)T$ are affected by the values of $\bm{w}(0,p)$ in the interval $p\in [-L,-L+\lambda_{\max}^-(H_1)T]$. Recall that $\bm{w}(0,p) = \psi \bm{u}_{f0}$, and $\psi(p) \neq e^{-p}$ for $p\in [-L, -L+\lambda_{\max}^-(H_1)T]$.
%For a rigorous proof, see Theorem 5.1,5.2.
%}
%\end{remark}
%{\color {green} $R+\lambda_{\max}^-(H_1)T$?}
%{\color{violet} I think we may not need to add a remark about the recovery interval here, as it may distract the reader; in this section, our focus in on the construction of $\psi$.}
\subsection{Construction of the function}

Motivated by the construction in the cut-off function method (cf. \eqref{cutoffmethod}), we define
\[\psi(p) = \phi(p) \e^{-p},\]
where the function $\phi(p)$ decays super-exponentially as  $p \to -\infty$; that is,
$\lim\limits_{p\to -\infty} \phi(p)\e^{-p} =0$, and satisfies $|\phi(p) - 1| \le \epsilon$ for $p \ge p_*$. Moreover, $\phi$ fulfills the following norm constraint
  \begin{equation}\label{phip}
  \|\phi^{(r)}\|_{L^2(\mathbb{R})}^{1/r} \le C  r, \qquad r\ge 1.
  \end{equation}

As analyzed in Section \ref{subsec:conditions}, the suboptimal cost of the cut-off method discussed in Section \ref{subsec:cut-off func} is primarily due to the lack of analyticity of the convolution kernel in \eqref{eq:mollifier1}. Specifically, the mollifier is not analytic, meaning it cannot be expressed as a convergent power series.
The motivation for using the mollifier as the convolution kernel is that it provides a smooth function with compact support, thus realizing the exact periodic boundary conditions.
However, if we relax this condition and seek a function that approximately satisfies the periodic boundary conditions, the most natural choice is the Gaussian $\e^{-p^2}$. In this case, a careful calculation reveals that
\begin{equation}\label{phip}
\phi(p) := (H(t) * \e^{-t^2} )(p) = \int_{\mathbb{R}} H(t)\e^{-(p-t)^2} \d t  = \frac{\text{erf}(p) + 1}{2},
\end{equation}
where $H(t)$ is the step function in \eqref{stepfunc}, and $\text{erf}(p)$ is the error function, which is defined as
\[\text{erf}(p) = \frac{2}{\sqrt{\pi}} \int_0^p \e^{-t^2} \d t.\]
We note that the error function is also used in \cite{Low25LCHS} to construct the optimal LCHS method, and our work is partially inspired by the approach presented in \cite{Low25LCHS}.
For later uses, we also introduce the complementary error function $\text{erfc}(p) = 1 - \text{erf}(p)$.

\begin{figure}[H]
  \centering
  \includegraphics[scale=0.5]{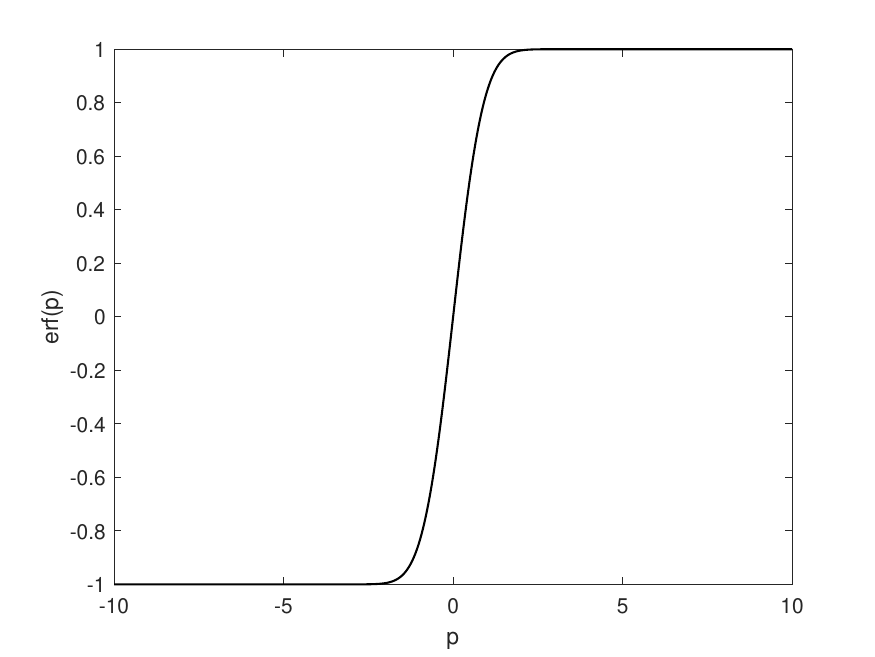}\\
  \caption{Error function $\text{erf}(p)$}\label{fig:erf}
\end{figure}

As shown in Fig.~\ref{fig:erf}, the error function is an odd function that approaches $\pm 1$ at a super-exponential rate (see Eq.~\eqref{erfbound}). This implies that $\phi(p)$ in \eqref{phip} tends to 1 at a super-exponential rate. However, this does not imply that $\psi(p) \approx 1$ for $p \ge p_*$, where $0 \le p_* \le \frac{1}{2}$. To address this, we rescale the error function and define
\begin{equation}\label{phip}
\phi(p) = \frac{\text{erf}(a p) + 1}{2},
\end{equation}
where $a \ge 1$ is a constant to be determined. Setting $x = ap$, we aim to find the lower bound of $x$ such that
\[
\Big| \frac{\text{erf}(x) + 1}{2} - 1 \Big| = \Big| -\frac{1}{2} \text{erfc}(x) \Big| \le \epsilon.
\]
For fixed $x>0$, we have the following expansion
\[\text{erfc}(x) = \frac{\e^{-x^2}}{x\sqrt{\pi}} \Big( 1 - \frac{1}{2x^2} + \frac{3}{4x^2} - \frac{15}{8x^6} + \cdots\Big),\]
which is an alternating series, leading to
\begin{equation}\label{erfbound}
\text{erfc}(x) \le \frac{\e^{-x^2}}{x\sqrt{\pi}} \le \e^{-x^2}, \qquad x\ge 1.
\end{equation}
This implies
\[ap = x \ge \log^{1/2} \frac{1}{\epsilon},\]
and thus, if we choose
\[a = 2 \log^{1/2} \frac{1}{\epsilon},\]
we obtain $p_* = 1/2$. Unless otherwise specified, we set $p_*=1/2$ throughout the following analysis.

\begin{remark}
Since the error function is odd and asymptotically approaches $ \pm 1 $, it follows that $ \phi(p) $ decays to zero on the negative axis at the same rate as shown in \eqref{erfbound}. This implies that $ \psi(p) $ decays super-exponentially as $p \to -\infty $ and exponentially as $p \to +\infty$.
\end{remark}

\begin{theorem} \label{thm:psibound}
Let $\phi(p)$ be defined in \eqref{phip} with $a = 2\log^{1/2} (1/\epsilon)$. Then the function $\psi(p) = \phi(p)\e^{-p}$ satisfies the required conditions \textbf{(H1)--(H3)}.
\end{theorem}

\begin{proof}
We only need to verify the inequality \eqref{psir}.  Let $\xi^{(k)}(p) = \phi^{(k)}(p) \e^{-p}$. Following the same calculation in \eqref{psirbound}, one gets
\begin{align*}
|\psi^{(r)}(p)|
 \le  2^r ( |\xi^{(0)}|^2 + \cdots + |\xi^{(r)}|^2 )^{1/2}.
\end{align*}
 This implies
\[ |\psi^{(r)}(p)|^{1/r} \lesssim \max_{0\leq k\leq r} |\xi^{(k)}|^{1/r} = \max_{0\le k\leq r} |\phi^{(k)}(p) \e^{-p}|^{1/r}
\lesssim  a \max_{0\leq k\leq r} |\text{erf}^{(k)}(ap) \e^{-p}|^{1/r},\]
since $\phi^{(k)}(p) = \frac12 a^k \text{erf}^{(k)}(ap)$.

A direct calculation gives
\[\text{erf}^{(k)}(ap) = \frac{2}{\sqrt{\pi}} (-1)^{k-1} H_{k-1}(ap) \e^{-a^2p^2}, \qquad 1\leq k\leq r,\]
where $H_k$ is the Hermitian polynomial,  defined by $H_k(p) = (-1)^k \e^{p^2} (\e^{-p^2})^{(k)}$, which leads to
\[|\text{erf}^{(k)}(ap) \e^{-p}|  \lesssim | H_{k-1}(ap) \e^{-a^2p^2-p}|.\]
It is known that
\[|H_k(ap)| \le C 2^{k/2} \sqrt{k!} \e^{a^2p^2
/2},\]
where $C \approx 1.086435$ (see \cite[Eq.~(7.66)]{Shenspectral} for example). Therefore, we obtain
\[\|\text{erf}^{(k)}(ap) \e^{-p}\|_{L^2(\mathbb{R})}^{1/r} \lesssim 2^{\frac{k}{2r}}(k!)^{1/(2r)}\|\e^{-a^2p^2/2-p}\|_{L^2(\mathbb{R})}  \le r^{1/2}, \quad k\leq  r,\]
which gives
\[\|\psi^{(r)}\|_{L^2(\mathbb{R})}^{1/r} \lesssim a r^{1/2}.\]
The proof is completed by noting that $a \simeq r^{1/2} \simeq \log^{1/2}(1/\epsilon)$.
\end{proof}

\begin{remark}
 Among the three conditions, the most challenging to satisfy is $\textbf{(H3)}$. In fact, functions satisfying $\textbf{(H3)}$ belong to the Gevrey class of order 1. Recall that a function $f \in C^\infty(\mathbb{R})$ is said to be in the Gevrey class $G^s(\mathbb{R})$ for $s \geq 1$ if, for every compact set $K \subset \mathbb{R}$, there exist constants $C, R > 0$ (independent of $r$) such that
\[\sup_{p\in \mathbb{R}}|f^{(r)}(p)| \leq C R^r(r!)^s\qquad \forall r\in \mathbb{N}.\]
    In particular, the case $s = 1$ corresponds to the Gevrey-1 class, which corresponds to the condition $\textbf{(H3)}$.

This definition implies that the construction of $\psi$ is not unique. For example, define
\[\chi(p) = \big(\int_{-\infty}^{\infty} \e^{-t^4}\d t\big)^{-1} \int_0^p \e^{-t^4}\d t,\quad
    \text{and} \quad
    \tilde{\phi}(p) = \chi(ap)+\frac{1}{2},\]
    where $a$ is an adjustable parameter chosen so that $\textbf{(H2)}$ is satisfied. Following the same line of the proof, we obtain a new initial function $\psi(p) = \tilde{\phi}(p)\e^{-p}$.
\end{remark}

\begin{remark}\label{remark}
   When there is no inhomogeneous term, the introduction of auxiliary constant variables is
   unnecessary. In this case, the total query complexity for $H_1$ and $H_2$ is
   $\tilde{\mathcal{O}}\big( \frac{\|\bm{u}(0)\|}{\|\bm{u}(T)\|} \alpha_H T \mu_{\max}\big)$.
   When the system is time-dependent, we employ the approach described in Section \ref{sub:time-dependent} to obtain a time-independent Hamiltonian system by introducing an additional dimension. Following the analysis in Theorem \ref{thm:complexity}, it is straightforward to derive that the query complexity for the time-dependent system is
    $	\tilde{\mathcal{O}}\big( \frac{\|\bm{u}(0)\|}{\|\bm{u}(T)\|} \alpha_H T \mu_{\max}\mu_{\max}^s\big)$, where $\mu_{\max}^s = \max_{l\in [N_s]}|\mu_l^s|$. After applying the same smoothing technique, the query complexity approaches $\widetilde{\mathcal{O}}(\frac{\|\bm{u}(0)\|}{\|\bm{u}(T)\|} \alpha_H T (\log \frac{1}{\varepsilon})^2)$, which includes an additional logarithmic term with respect to $\varepsilon$, consistent with the observation in \cite{ACL2023LCH2} (see Table 1 there).
\end{remark}

\begin{remark}
Recent results in \cite{ALWZ25} establish quantum query lower bounds for solving ODEs by reduction to quantum linear system solvers.
These bounds show that any algorithm requires at least $\Omega(T^\alpha)$ queries for sufficiently long evolution time $T$ and at least $\Omega(\log^\alpha(1/\varepsilon))$ queries for target precision $\varepsilon$ in the worst case, where $\alpha$ is the query lower bound exponent for quantum linear system solvers. According to \cite{Costa2021QLSA}, one has $\alpha = 1$. This implies that our Schr\"odingerization algorithm achieves optimal dependence on matrix queries.
\end{remark}

\section{Error estimate for the Schr\"odingerization method} \label{sec:err}

%%This section provides a proof of the claim below Eq.~\eqref{errw}.
%({\color {green} For such an important statement, it should be listed as a Theorem or something like that. It is not good to say we will prove something below some equation. It is hard for readers to find that statement. }) {(\color{cyan} We summarize the main results in a theorem. This section is devoted to error estimates, which were entirely absent in our previous submission. In our response to the reviewers, we plan to include this subsection as supplementary material and upload it to arXiv.)}
%%In this section, we focus on elucidating the theoretical guarantees for the recovery of the target variable and the numerical discretization error for smooth initial functions.
%%satisfying conditions \textbf{(H1)--(H2)} and \eqref{mumax}. ({\color {green}  Again, \eqref{mumax} is important. It should not be just labeled by an equation number.})

In this section, we analyze the error of the Schr\"odingerization method with smooth initial functions. The discretization error consists of two parts: recovering the target variable
$\bm{u}_{f}$, the solution of \eqref{eq: homo Au}, from
$\bm{w}$, the solution of \eqref{u2v}; and the numerical error from applying the Fourier spectral method to \eqref{u2v}. For simplicity, we assume that both $A$ and the source term
$\bm{b}$ are time-independent; the time-dependent case can be treated analogously.

% Since $H_1$ is Hermitian, there exits a unitary $U$ such that
% $H_1 = U \Sigma U^{\dagger}$, where $\Sigma =\text{diag}\{\lambda_1, \cdots, \lambda_N\}$.
% Introducing the change of variable $\bm{w}_u = U^{\dagger}\bm{w}$, one gets the equation of $\bm{w}_u $:
% \begin{equation*}
% \frac{\d}{\d t} \bm{w}_u =  \Sigma \partial_p \bm{w}_u +\i H_{u}\bm{w}_u, \quad \bm{w}_u(0,p) = \psi(p)U^{\dagger}\bm{u}_{f0},
% \end{equation*}
% where $H_{u} = U H_2 U^{\dagger}$ is Hermitian.
%Therefore, without loss of generality, we assume that $H_1$ is a diagonal matrix with the $i$-th diagonal entry denoted by $\lambda_i$.
%According to Theorem 3, for the inhomogeneous system \eqref{ODElinear}, the eigenvalues of $H_1$ still include both positive and negative values.

\subsection{Recovery from \eqref{u2v} with smooth initializations}

%When all eigenvalues of $H_1$ are negative, the convection term in equation \eqref{u2v} corresponds to a wave propagating from right to left. In this case, the values of $\psi(p)$ for $p \leq 0$ do not influence the behavior of $\bm{w}(t,p)$ within the region $p \in (0,R)$ and $t \in (0,T)$.
%However, if $\lambda_i(H_1) > 0$ for some $i$, certain wave components governed by \eqref{u2v} will propagate from left to right. It is essential to note that any waves associated with positive eigenvalues of $H_1$ that originate from $p < 0$ and travel rightward will introduce spurious contributions into the domain $p > 0$. Therefore, to accurately recover $\bm{u}(t)$ from $\bm{w}(t,p)$, one must restrict the consideration to values of $p$ satisfying $ \lambda_{\max}^+(H_1) T <p < R-\lambda_{\max}^-(H_1)T$, thereby avoiding contamination by these artificial wave artifacts.

In the following, we provide a rigorous basis for the recovery of $\bm{u}_f(t)$ from $\bm{w}(t,p)$ under the stated conditions.

\begin{lemma}\label{lem:contr_by_init}
	Let $\bm{\omega}:[0,T]\times\mathbb R \to \mathbb C^N$ be a solution to
	\[
	\frac{\d \bm{\omega}}{\d t} = -H_1 \partial_p \bm{\omega} + \i H_2 \bm{\omega},\quad
	\bm{\omega}(0,p) = \begin{cases}
		\textbf{0},   \quad p\in [p_*,R] \\
		\bm{\omega}_0(p),  \quad p\in \bbR \backslash [p_*,R],
	\end{cases}
	\]
where $\bm{\omega}_0(p)$ is a given function. Then we have
\[\bm{\omega}(T,p) = \bm{0}, \qquad  p\in (p_*+\lambda_{\max}^+(H_1)T,~~R-\lambda_{\max}^-(H_1)T).\]
\end{lemma}
\begin{proof}	
We first derive a local conservation law for the squared amplitude.
Since $H_1,H_2$ are Hermitian, from the evolution equation $\partial_t \bm{\omega} = -H_1 \partial_p \bm{\omega} + \i H_2 \bm{\omega}$, we obtain
\begin{equation}\label{eq:cons}
\partial_t \big(\bm{\omega}^\dagger \bm{\omega}\big)
= \bm{\omega}^\dagger \partial_t \bm{\omega} + (\partial_t \bm{\omega})^\dagger \bm{\omega}
=  -\partial_p \big( \bm{\omega}^\dagger H_1 \bm{\omega}\big).
\end{equation}
%Substituting the evolution equation
%$\partial_t \bm{\omega} = -H_1 \partial_p \bm{\omega} + iH_2 \bm{\omega}$,
%we obtain
%\[
%\partial_t \big(\bm{\omega}^\dagger \bm{\omega}\big)
%= -\partial_p \big( \bm{\omega}^\dagger H_1 \bm{\omega}\big),
%\]
%because the skew-Hermitian terms involving $iH_2$ cancel.
%Thus, the solution satisfies the conservation law
%\begin{equation}\label{eq:cons}
%	\partial_t \big(\bm{\omega}^\dagger \bm{\omega}\big)
%	+ \partial_p \big(\bm{\omega}^\dagger H_1 \bm{\omega}\big) = 0.
%\end{equation}
%
%Let $\lambda_{\max}^+(H_1) := \max\{0,\,\max \sigma(H_1)\}$ and
%$\lambda_{\max}^-(H_1) := \max\{0,\,-\min \sigma(H_1)\}$,
%where $\sigma(H_1)$ denotes the spectrum of $H_1$.
For $t\in[0,T]$, define the moving interval
\[
I(t) = \big(p_*+\lambda_{\max}^+(H_1) t,\;\; R-\lambda_{\max}^-(H_1) t\big),
\]
and the localized energy
\[
E_{\omega}(t) := \int_{I(t)} \|\bm{\omega}(t,p)\|^2 \,\mathrm{d}p.
\]
By the Leibniz's rule and \eqref{eq:cons}, one finds
\begin{align*}
\frac{\mathrm{d}}{\mathrm{d}t}E_{\omega}(t)
= & -\Big[\,\bm{\omega}^\dagger H_1 \bm{\omega}\,\Big]_{p=p_*+\lambda_{\max}^+(H_1) t}^{R-\lambda_{\max}^-(H_1) t}+ \lambda_{\max}^-(H_1) \|\bm{\omega}(t,R-\lambda_{\max}^-(H_1) t)\|^2 \\
&\quad - \lambda_{\max}^+(H_1) \|\bm{\omega}(t,p_*+\lambda_{\max}^+(H_1) t)\|^2\leq 0.
\end{align*}
At $t=0$, the initial condition ensures $\bm{\omega}(0,p)\equiv \bm{0}$ for $p\in(p_*,R)$, hence $E(0)=0$.
By monotonicity, we conclude $E_{\omega}(t)\equiv 0$ for all $t\in[0,T]$.
Therefore, $\bm{\omega}(t,p)\equiv 0$ almost everywhere in $I(t)$, and by continuity, it follows that
$\bm{\omega}(T,p)\equiv \bm{0}$ for $ p \in \big(p_*+\lambda_{\max}^+(H_1) T,\;\; R-\lambda_{\max}^-(H_1) T\big)$.
\end{proof}

\begin{remark}\label{rmrk}
	Following the proof, it is clear that if $\bm{\omega}(0,p) = \epsilon \bm{u}_{I}$ for $p\in [p_*,R]$, then
	\[
	 \int_{p_*+\lambda_{\max}^+(H_1)T}^{R-\lambda_{\max}^-(H_1)T}\|\bm{\omega}(T,p)\|^2 \d p
	\leq  \int_{p_*}^{R}\|\bm{\omega}(0,p)\|^2\d p =  \epsilon^2 \|\bm{u}_{I}\|^2 (R-p_*).
	\]
	
\end{remark}

\begin{theorem}\label{thm:recovery1}
	Assume $\psi(p) \in L^2(\bbR)$ with $\psi(p)=\e^{-p}$ in $(0,R)$, and $R>(\lambda_{\max}^-(H_1) + \lambda_{\max}^+(H_1))T$. Then the solution of \eqref{ODElinear} can be recovered by
	\[\bm{u}_f(T) = \e^{p}\bm{w}(T,p),\quad  p\in (\lambda_{\max}^+(H_1)T, ~~R-\lambda_{\max}^{-}(H_1)T) ,\]
	where $\bm{w}(T,p)$ is the solution to \eqref{u2v}.
\end{theorem}
\begin{proof}
The function $\widetilde{\bm{w}}(t,p)=\e^{-p}\e^{(H_1+\i H_2)t}\bm{u}_{0}$ satisfies both the PDE and the initial condition whenever the backward characteristics from $(t,p)$ remain inside $(0,R)$, i.e., $p-\lambda_j t \in (0, R)$, where $\lambda_j$ is the eigenvalue of $H_1$.
This requires $\lambda_{\max}^{+}(H_1)T<p<R-\lambda_{\max}^-(H_1)T$.
	Define the error vector $\bm{e}_w = \bm{w} - \widetilde{\bm{w}}$. Then $\bm{e}_w(t,p)$ satisfies the same PDE, and $\bm{e}_w(0,p) \equiv \bm{0}$ for all $p \in (0, R)$. The proof is completed by applying Lemma~\ref{lem:contr_by_init}.
\end{proof}
\begin{theorem}\label{thm:recovery2}
Assume that $\psi(p) \in L^2(\bbR)$  satisfies condition \textbf{(H2)}.
The recovery from $\bm{w}(T,p)$,  which is the solution to \eqref{u2v},  is defined by
	\[\bm{u}_f^*(T,p) = \e^p \bm{w}(T,p) \qquad  p \in (p_*+\lambda_{\max}^+(H_1)T,~~R-\lambda_{\max}^-(H_1)T).\]
Assume $p_*=\frac{1}{2}$ and $R = \lambda_{\max}^-(H_1)T+\mathcal{O}(1)$. The $L^2$ error estimate between the recovery and the solution to  \eqref{ODElinear} is given by
	\[
	\frac{1}{L_*} \int_{p_*+\lambda_{\max}^+(H_1)T}^{R-\lambda_{\max}^-(H_1)T}\|\bm{u}_f^*(T,p)-\bm{u}_f(T)\|^2 \d p \lesssim  \epsilon^2 \|\bm{u}_{I}\|^2,
	\]
where $L_* = R-p_*-(\lambda_{\max}^+(H_1)T+\lambda_{\max}^-(H_1)T)$ is the length of the recovery domain.
\end{theorem}
\begin{proof}
     Since the matrix $A$ and the vector $\bm{b}$ in system \eqref{ODElinear} are time-independent and satisfy $\lambda(A + A^\dagger) \leq 0$, it follows that $\lambda_{\max}^+(H_1)T \leq 1/2$.
     The proof is completed by applying Theorem~\ref{thm:recovery1} and Remark~\ref{rmrk}.
\end{proof}

\subsection{Error estimate of spectral discretization}
In this subsection, we provide the detailed proof of Theorem \ref{thm:err w-wh}.
The error associated with \eqref{errw} arises from two sources: (i) the truncation error introduced by restricting the equation for $\bm{w}$ from the whole space to a finite domain with periodic boundary conditions, and (ii) the discretization error of the Fourier spectral method applied to the periodic problem on the finite domain.
\subsubsection{Error estimate of the truncation}
The vector $\bb{w}(\cdot, p)$ for Eq.~\eqref{u2v} is defined in $\mathbb{R}$. However, in the implementation, $\mathbb{R}$ is truncated to the interval $[-L, R]$, with $L, R$ satisfying \eqref{eq: L,R,criterion},  resulting in the solution $\bm{w}_h(\cdot, p)$ in \eqref{heatww}. In the following, we will characterize the truncation error associated with this approximation.

It is apparent that the discretization \eqref{heatww} serves as an approximation for the following system with periodic boundary conditions:
\begin{equation}\label{eq:periodic w}
	\begin{cases}
		\frac{\partial}{\partial t} \mathcal{W}= -H_1\partial_p \mathcal{W} + \i H_2 \mathcal{W}, \quad  0< t < T,~~-L< p < R,\\
		\mathcal{W}(t,-L) = \mathcal{W}(t,R),\\
		\mathcal{W}(0,p) = \psi(p)\bm{u}_{I},
	\end{cases}
\end{equation}
where $\psi(p)$ satisfies $\textbf{(H1)}-\textbf{(H2)}$, and $\psi \in H^r((-L,R))$.
Noting that $\psi^{(k)}(p)\approx 0$ at $p =-L,R$ for $k\le r$, it implies that each entry of $\mathcal{W}(0,p)$ can be treated as a function in $H_p^k[-L,R]$ for any $0\le k\le r$, which consists of functions with derivatives of order up to $(k-1)$ being periodic on $[-L,R]$.
It is important to note that this estimate applies not only to the initial data but also to the solution of \eqref{eq:periodic w}, since $\mathcal{W}$ satisfies a transport equation in the $p$ direction thus preserves the regularity in $p$ in the initial data as time evolves.

We observe that to estimate the error between $\bm{w}_h$ and $\bm{w}$, it suffices to bound the error between $\mathcal{W}$ and $\bm{w}$.

\begin{lemma}
	\label{lem:boundary-control}
	Let $\bm{w}:[0,T]\times\mathbb R \to \mathbb C^N$ be the solution to
	\eqref{u2v} and $\psi(p)$ satisfies $\textbf{(H1)}$.
	Then  it holds
	\begin{align}
		\left|\int_0^T \bm{w}^{\dagger}(t,-L) H_1 \bm{w}(t,-L)\d t \right|&\lesssim \,  \int_{-\infty}^{-L+\lambda_{\max}^-(H_1)T}
		\|\bm{w}(0,p)\|^2\d p,  \label{eq:W Left}\\
		\left| \int_0^T\bm{w}^{\dagger}(t,R) H_1 \bm{w}(t,R)  \d t \right| &\lesssim\, \int_{R-\lambda_{\max}^+(H_1)T}^{+\infty}
		\|\bm{w}(0,p)\|^2\d p  . \label{eq:W Right}
	\end{align}
\end{lemma}

\begin{proof}
	Define the local energy density $E(t,p)=\|\bm{w}(t,p)\|^2=\bm{w}(t,p)^\dagger \bm{w}(t,p)$.
	Multiplying \eqref{u2v} by $\bm{w}^\dagger$ from the left
	and using the Hermiticity of $H_2$, we obtain the conservation law
	\begin{equation}\label{eq:energy-cons}
		\partial_t E(t,p) + \partial_p \big(\bm{w}^\dagger H_1 \bm{w}\big) = 0.
	\end{equation}
	Let $\{\chi_n\}_{n\ge1}\subset C_c^\infty(\mathbb R)$ be a sequence of smooth cutoff functions such that
	$\chi_n(x)=1$ for $x\le 0$, $\chi_n(x)=0$ for $x\ge 1/n$, and $\chi_n'(x)\le0$.
	For fixed $t\in[0,T]$, define the moving cutoff
	\[
	\Phi_{n,s}(p):=\chi_n\big(p+L-\lambda_{\max}^-(H_1)(t-s)\big), \qquad s\in[0,t],
	\]
	and the localized energy
	\[
	I_n(s):=\int_{\mathbb R} \Phi_{n,s}(p)\, E(s,p)\, \d p.
	\]
	
	Differentiating $I_n(s)$ with respect to $s$, and using \eqref{eq:energy-cons}
	together with integration by parts, one obtains
	\[
	I_n'(s) = \int_{\mathbb R} \chi_n'\left(p+L-\lambda_{\max}^-(H_1)(t-s)\right) \, \big(\lambda_{\max}^-(H_1)E(s,p) +\bm{w}^\dagger H_1 \bm{w}(s,p)\big)\, \d p.
	\]
	The boundary terms from integration by parts vanish--after ignoring the error of $O(\epsilon)$-- due to the exponential decay of $\bm{w}$.
	Since $\chi_n'\le 0$ and $-\lambda_{\max}^-(H_1)E \leq \bm{w}^\dagger H_1 \bm{w}\leq \lambda_{\max}^+(H_1) E$,
	the integrand is positive, hence $I_n'(s)\leq 0$ for all $s\in[0,t]$.
	Therefore $I_n(t)\le I_n(0)$.
	
	By the construction,
	\[
	I_n(t)=\int_{\mathbb R} \chi_n(p+L)\,\|\bm{w}(t,p)\|^2 \d p,
	\qquad
	I_n(0)=\int_{\mathbb R} \chi_n(p+L-\lambda_{\max}^-(H_1)t))\,\|\bm{w}(0,p)\|^2 \d p.
	\]
Taking $n\to\infty$ , we obtain
	\[
	\int_{-\infty}^{-L}\|\bm{w}(t,p)\|^2 \d p \;\leq\; \int_{-\infty}^{-L+\lambda_{\max}^-(H_1)T}\|\bm{w}(0,p)\|^2 \d p.
    %=
%	\int_{-L}^{-L+\lambda_{\max}^-(H_1)T}  \|\bm{w}(0,p)\|^2 \d p .
	\]
	Integrating \eqref{eq:energy-cons} in the time-space domain $ (0,T)\times (-\infty,-L)$ gives
	\begin{align}
		&\int_{-\infty}^{-L} \|\bm{w}(T,p)\|^2 \d p
        -\int_{-\infty}^{-L} \|\bm{w}(0,p)\|^2\d p  \notag \\
        =& -\int_0^T \int_{-\infty}^{-L} \partial_p\big( \bm{w}^{\dagger} H_1 \bm{w}\big) \d p \d t= -\int_0^T  \bm{w}^{\dagger}(t,-L) H_1 \bm{w}(t,-L) \d t
	\end{align}
Thus, we have
	\begin{equation}
		\left|\int_0^T  \bm{w}^{\dagger}(t,-L) H_1 \bm{w}(t,-L) \d t\right| \leq 2 \int_{-\infty}^{-L+\lambda_{\max}^-(H_1)T} \|\bm{w}(0,p)\|\d p.
    \end{equation}
	The proof for \eqref{eq:W Right} is similar, which is omitted here.
\end{proof}

\begin{lemma}\label{lem:boundary-trace}
$\bm{w}:[0,T]\times\mathbb R \to \mathbb C^N$ be the solution to
	\eqref{u2v} and $\psi(p)$ satisfies $\textbf{(H1)}$.
	Then  it holds that
\begin{equation}\label{eq:wLR}
\int_0^T \|\bb{w}(t,-L)\|^2\,\mathrm{d}t
+\int_0^T \|\bb{w}(t,R)\|^2\,\mathrm{d}t
\;\lesssim\; T\,\epsilon^{2}\,\|\bb{u}_I\|^{2},
\end{equation}
where the hidden constant is independent of $H_1$ and $\bb{u}_I$.
\end{lemma}

\begin{proof}
Let $U(t)$ be the solution of the matrix ODE
\[
\partial_t U(t) = \mathrm{i}H_2(t)U(t),\qquad U(0)=I.
\]
Since $H_2(t)$ is Hermitian, $U(t)$ is unitary for all $t\in[0,T]$.
Define
\[
\bb{v}(t,p):=U(t)^{-1}\bb{w}(t,p).
\]
Then $\|\bb{v}(t,p)\|=\|\bb{w}(t,p)\|$ and $\bb{v}$ solves
\[
\partial_t \bb{v}(t,p) = -\widetilde H_1\,\partial_p \bb{v}(t,p),
\qquad
\bb{v}(0,p)=\psi(p)\,\bb{u}_I,
\]
where $\widetilde H_1 := U(t)^{-1}H_1U(t)$ is Hermitian and has the same eigenvalues as $H_1$.
Since $H_1$ is time-independent, the spectrum of $\widetilde H_1$ is independent of $t$.

Diagonalizing $\widetilde H_1=Q\Lambda Q^*$ with a unitary $Q$ and 
$\Lambda=\mathrm{diag}(\lambda_1,\dots,\lambda_N)$, we set
\[
\bb{z}(t,p):=Q^*\bb{v}(t,p),\qquad \tilde{\bb{u}}_I:=Q^*\bb{u}_I.
\]
Then each component $z_j$ satisfies a scalar transport equation
\[
\partial_t z_j(t,p) = -\lambda_j\,\partial_p z_j(t,p),\qquad
z_j(0,p)=\psi(p)\,\tilde u_{I,j},
\]
whose solution along characteristics is
$
z_j(t,p) = \psi(p+\lambda_j t)\,\tilde u_{I,j}.
$

Evaluating at the boundaries $p=-L$ and $p=R$ gives
\[
z_j(t,-L)=\psi(-L+\lambda_j t)\,\tilde u_{I,j},\quad
z_j(t,R)=\psi(R+\lambda_j t)\,\tilde u_{I,j}.
\]
By the choice of $L,R$ and the definition of $\lambda_{\max}^{\pm}(H_1)$,
for every $j$ and $t\in[0,T]$ the characteristic footpoints satisfy
\[
-L+\lambda_j t\in(-\infty,-L+\lambda_{\max}^-(H_1)T],\quad
R+\lambda_j t\in[R-\lambda_{\max}^+(H_1)T,+\infty).
\]
Hence the decay condition \textbf{(H1)} yields
\[
|\psi(-L+\lambda_j t)|\le 2\epsilon,\qquad
|\psi(R+\lambda_j t)|\le 2\epsilon,
\quad 0\le t\le T,\ \forall j.
\]
Therefore
\[
|z_j(t,-L)|^2\le 4\epsilon^2|\tilde u_{I,j}|^2,\qquad
|z_j(t,R)|^2\le 4\epsilon^2|\tilde u_{I,j}|^2.
\]

Summing over $j$ and using the unitarity of $Q$ and $U(t)$, we obtain
\[
\|\bb{w}(t,-L)\|^2
= \|\bb{z}(t,-L)\|^2
= \sum_{j=1}^N |z_j(t,-L)|^2
\le 4\epsilon^2\sum_{j=1}^N |\tilde u_{I,j}|^2
= 4\epsilon^2\|\bb{u}_I\|^2,
\]
and similarly
$
\|\bb{w}(t,R)\|^2\le 4\epsilon^2\|\bb{u}_I\|^2 $ for $0\le t\le T.
$
Integrating in time yields \eqref{eq:wLR}.
\end{proof}

Similarly, the estimate for the periodic case follows an argument analogous to the previous proof; we therefore state it without proof in the following lemma.
\begin{lemma}\label{lem:boundary control periodic}
    Assume $\mathcal{W}$ is the solution to \eqref{eq:periodic w} with periodic boundary conditions. Then it holds
    \begin{equation}
        \left|\int_0^T \mathcal{W}^{\dagger}(t,-L) H_1 \mathcal{W}(t,-L)\d t \right|\lesssim \,  \int_{-L}^{-L+\lambda_{\max}^-(H_1)T}
		\|\mathcal{W}(0,p)\|^2\d p + \int_{R-\lambda^+_{\max}(H_1)T}^{R}
		\|\mathcal{W}(0,p)\|^2\d p.
    \end{equation}
\end{lemma}
\begin{lemma}\label{lem:boundary control periodic}
    Assume $\mathcal{W}$ is the solution to \eqref{eq:periodic w} with periodic boundary conditions. Then it holds
\begin{equation}\label{eq:peri wLR}
\int_0^T \|\mathcal{W}(t,-L)\|^2\,\mathrm{d}t
+\int_0^T \|\mathcal{W}(t,R)\|^2\,\mathrm{d}t
\;\lesssim\; T\,\epsilon^{2}\,\|\bb{u}_I\|^{2}.
\end{equation}
\end{lemma}
Then, we get the error between $\bm{w}(t,p)$ and $\mathcal{W}(t,p)$.
\begin{theorem}\label{thm:truncation err}
	Let $\mathcal{E}(t,p) = \bm{w}(t,p) - \mathcal{W}(t,p)$. Assume $\psi$ satisfies \textbf{(H1)}.
	It follows that
	\begin{equation*}
		\|\bm{w}(T,p)-\mathcal{W}(T,p)\|_{L^2((-L,R))} \lesssim \epsilon \|\bm{u}_{I}\|.
	\end{equation*}
\end{theorem}

\begin{proof}
	The error function $\mathcal{E}(t,p) $ satisfies
	\begin{equation}\label{eq:err E}
		\frac{\d}{\d t} \mathcal{E} = -H_1\partial_p \mathcal{E} + \i H_2 \mathcal{E},\quad
		\mathcal{E}(0,p) = \bm{0}.  %\mathcal{E}(t,-L) = \bm{w}(t,-L),\quad \mathcal{E}(t,R) = \bm{w}(t,R).
	\end{equation}
As in \eqref{eq:cons}, by testing \eqref{eq:err E} against $\mathcal{E}^{\dagger}$ and integrating by parts, we arrive at
	\begin{equation*}
		\begin{aligned}
			\|\mathcal{E}(T,p)\|_{L^2((-L,R))}^2  \lesssim &
			\left|\int_0^T \bm{w}^{\dagger}(t,- L) H_1 \bm{w}(t,-L)\d t\right|
			+\left|\int_0^T \bm{w}^{\dagger}(t,R) H_1 \bm{w}(t,R)\d t\right|\\
            &+\left|\int_0^T \mathcal{W}^{\dagger}(t,L) H_1 \mathcal{W}(t,L)\d t\right|  
            + \frac{1}{2T} \int_0^T \|\bm{w}(t,-L)\|^2\,\d t \notag \\
            &+\frac{1}{2T} \int_0^T \|\bm{w}(t,R)\|^2\,\d t
            + \frac{1}{2T} \int_0^T \|\mathcal{W}(t,-L)\|^2\,\d t 
            +\frac{1}{2T} \int_0^T \|\mathcal{W}(t,R)\|^2\,\d t.
		\end{aligned}
	\end{equation*}
	which completes the proof by applying Lemma \ref{lem:boundary-control} -- \ref{lem:boundary control periodic}.
\end{proof}

\subsubsection{Main proof of Theorem \ref{thm:err w-wh}}

For brevity, we set $a = -L$ and $b = R$.
Define the complex $N_p$-dimensional space
\[X_{N_p}:= \text{span}\big\{\phi_l(p) = \e^{\i\mu_l(p-a)}:0\leq l<N_p\big\}, \]
where $\mu_l = \frac{2\pi}{b-a}(l-\frac{N_p}{2})$.
Let $\bb{w}(t,p) = [w_1(t,p), \cdots, w_{2N}(t,p)]^\top$ be the solution to \eqref{u2v}.
 The approximate solution $\bm{w}_h\in (X_{N_p})^{2N}$ is then given  in \eqref{interpcoeff}.
 %by $\bb{w}_h(t,p) = [w_{1,h}(t,p),\cdots, w_{N,h}(t,p)]^\top \in (X_{N_p})^{2N}$ from the numerical solution, which can be written as
 % \begin{equation}\label{interpcoeff}
 % \bb{w}_h(t,p) = \sum_{l=0}^{N_p-1} \tilde{\bb{w}}_{l,h}(t) \phi_l(p), \qquad
 % \tilde{\bb{w}}_{l,h}(t) = \frac{1}{N_p} \sum\limits_{k=0}^{N_p-1} \big((\bra{k}\otimes I)\bm{W}_h\big) \e ^{ - \i \mu_l (p_k-a)},
 % \end{equation}
 % where $\bm{W}_h$ is obtained by \eqref{heatww}.
 % From the definition of Fourier transform, it can be seen that
 % \[\tilde{\bb{W}}_h(t) = (\Phi^{-1}\otimes I)\bm{W}_h(t),\qquad \tilde{\bb{W}}_h(t)
 % = [\tilde{\bb{w}}_{0,h}(t) ; \cdots; \tilde{\bb{w}}_{N_p-1,h}(t)],\]
 % with ``;'' indicating the straightening of $\{\tilde{\bb{w}}_{i,h}\}_{i\ge 1}$ into a column vector.

Define the $L^2$-orthogonal projection $P_h:L^2((a,b))\to X_{N_p}$ , given by
\[P_h u =\sum_{k=0}^{N_p-1}\hat{u}_l \phi_l(p),\qquad
\hat{u}_l = \frac{1}{b-a}\int_{a}^b u\e^{-\i\mu_l(p-a)}\d  p.
\]
Next, we consider the Fourier interpolation denoted by $\Pi$, that is
\[\Pi u (p) = \sum_{l=0}^{N_p} \tilde{u}_l \phi_l(p),\quad
\tilde{u}_l = \frac{1}{N_p} \sum_{k=0}^{N_p} u(p_k)\e^{-\i\mu_k(p_l-a)},\quad 0\leq l<N_p.\]
Then, one has the following estimates for $u\in H^r_p((a,b))$ \cite{Shenspectral}
\begin{equation}\label{eq:err inter}
    \|P_h u - \Pi u\|_{L^2((a,b))} \lesssim (\triangle p)^r \|u^{(r)}\|_{L^2((a,b))}.
\end{equation}

Using the triangle inequality, one has
\begin{equation}\label{eq:err w-wh}
\|\bb{w}(\cdot,p) - \bb{w}_h(\cdot,p)\|_{L^2((a,b))} \leq \|\bb{w}(\cdot,p) - \mathcal{W}(\cdot,p)\|_{L^2((a,b))}+\|\mathcal{W}(\cdot,p) - \bb{w}_h(\cdot,p)\|_{L^2((a,b))}.
\end{equation}
According to Theorem \ref{thm:truncation err}, it is sufficient to prove the second term.
By the triangle inequality, the second part of \eqref{eq:err w-wh} can be split as
\[\| \mathcal{W}(\cdot,p) - \bb{w}_h(\cdot,p)\|_{L^2((a,b))}
	\le \|\mathcal{W}(\cdot,p) - P_h\mathcal{W}(\cdot,p)\|_{L^2((a,b))} + \|P_h \mathcal{W}(\cdot,p) - \bb{w}_h(\cdot,p)\|_{L^2((a,b))}.\]
For the first term on the right-hand side, the standard projection error estimate for the discrete Fourier transform yields
\[\|\mathcal{W}-P_h \mathcal{W}(\cdot,p)\|_{L^2((a,b))} \lesssim \Big(\frac{b-a}{N_p}\Big)^k \|\mathcal{W}^{(k)}(\cdot,p)\|_{L^2((a,b))}
\lesssim \Big(\frac{b-a}{N_p}\Big)^k \|\psi^{(k)}\|_{L^2((a,b))} \|\bb{u}_{I}\|.\]
The second inequality follows by expanding $\mathcal W$ and $\psi$ in the periodic Fourier basis and observing that, for each Fourier mode, the evolution operator $e^{\,\mathrm{i}(-\xi H_1 + H_2)t}$ is unitary, so that $\|\partial_p^k \mathcal W(t,\cdot)\|_{L^2((a,b))}$ is controlled by $\|\psi^{(k)}\|_{L^2((a,b))}\,\|\mathbf{u}_I\|$ with a constant independent of $t$.
For the second term, by definition, one gets
\[P_h \mathcal{W}(\cdot,p) - \bb{w}_h(\cdot,p) = \sum\limits_{l= 0}^{N_p-1} (\hat{\bb{w}}_l(t)- \tilde{\bb{w}}_{l,h}(t)) \phi_l(p), \]
where $\hat{\bm{w}}_l = \int_{a}^b \bm{w}e^{-i\mu_l(p-a)}\d p/(b-a)$, which gives
\[ \|P_h \mathcal{W}(\cdot,p) - \bb{w}_h(\cdot,p)\|_{L^2((a,b))}^2 = \sum\limits_{l= 0}^{N_p-1} \|\hat{\bb{w}}_l(t)- \tilde{\bb{w}}_{l,h}(t)\|^2. \]
Let $\hat{\bm{e}}_l = \hat{\bm{w}}_l - \tilde{\bm{w}}_{l,h}$. It is easy to check that
\[\begin{cases}
\dfrac{\d }{\d t} \hat{\bb{e}}_l(t) = -\i (\mu_l H_1 -H_2)  \hat{\bb{e}}_l(t)\\
\hat{\bb{e}}_l(0) = (\hat{\psi}_l - \tilde{\psi}_l) \bb{u}_{I}
\end{cases} \qquad l = 0,\cdots, N_p-1.\]
Therefore, we have
\begin{align}
\|P_h \mathcal{W}(\cdot,p) - \bb{w}_h(\cdot,p)\|_{L^2((a,b))}^2
&=\sum_{l=0}^{N_p-1} \|\hat{e}_l\|^2 = \sum_{l=0}^{N_p-1} |\hat{\psi}_l-\tilde{\psi}_l|^2\|\bm{u}_{0}\|^2   \notag  \\
&=\|P_h\psi - \Pi \psi\|^2_{L^2((a,b))}\|\bm{u}_{I}\|^2.
\end{align}
The proof is finished by using \eqref{eq:err inter} and \eqref{mumax}.

\section{Discussion}
The LCHS method proposed in \cite{ALL2023LCH} is closely related to the Schr\"odingerization framework.
%In fact, the LCHS method can be seen as a continuous treatment in the auxiliary variable $p$ ({\color {green} Well, the Schr\"odingerization treats $p$ continuously already, so I don't know what you meant to say here}).
For example, by taking $\psi(p)=\e^{-|p|}$ in \eqref{u2v}, one applies the continuous Fourier transform on $p$ to get the Schr\"odinger-type system
\[
\frac{\d}{\d t} \hat{\bm{w}} = \i(\xi H_1+H_2)\hat{\bm{w}},
\qquad \hat{\bm{w}}(0)= \frac{1}{\pi (1+\xi^2)}\bm{u}_{I}.
\]
This, together with the recovery formula, yields
\[
\bm{u} = \bm{w}(t,0) = \int_{\mathbb{R}} \frac{1}{\pi (1+\xi^2)}
\mathcal{T}\exp\!\Big(\i\int_0^t \xi H_1(s)+H_2(s)\,\d s\Big)\bm{u}_{I} \, \d \xi,
\]
when the eigenvalues of $H_1$ are non-positive, where $\mathcal{T}\exp$ is the time-ordering exponential operator,
which is consistent with the exact representation in \cite[Theorem~1]{ALL2023LCH}.
Therefore, the two approaches share the same foundation at the analytical level.

Due to the same foundation, the LCHS in \cite{ALL2023LCH} is also a first-order method. Subsequent improvements on the precision are presented in \cite{ACL2023LCH2}, with their relation to our smooth extension covered in Section \ref{subsec:FourierSmooth}. The differences between the Schr\"odingerization and the LCHS arise in the discretization strategies applied to the auxiliary variable $p$. In particular, Schr\"odingerization leads to the matrix-query complexity of order
\[
\mathcal{O}\!\big(\log(1/\varepsilon)\big),
\]
whereas the optimized LCHS method in \cite{ACL2023LCH2} achieves
\[
\mathcal{O}\Big((\log(1/\varepsilon))^{1/\beta}\Big), \qquad \beta \in (0,1).
\]
in the time-dependent case. Clearly, the latter one has sub-optimal dependence on matrix queries since $\beta$ cannot be exactly equal to 1. 

During the revision stage of this work, we were informed of  a recent study \cite{Low25LCHS}, which introduces a generalized \textit{approximate version} of LCHS. This extension incorporates a kernel function with exponential decay -- which inspired our own work -- also allowing for a quantum ODE algorithm that can achieve optimal precision dependence.

%\section*{CRediT authorship contribution statement}
%
%Shi Jin, Nana Liu and Yue Yu all contributed to the conceptualization, methodology and writing.
%
%\section*{Declaration of competing interest}
%
%The authors declare that they have no known competing financial interests or personal relationships that could have
%appeared to influence the work reported in this paper.
%
%\section*{Data availability}
%
%No data was used for the research described in the article.
%
%
\section*{Acknowledgements}
%{\color{violet}
%In the earlier version of this paper, we believed that the cut-off function method could achieve optimal precision dependence. We would like to thank Dr. Guang Hao Low for pointing out two critical errors in our analysis. The first error concerned the hidden constant in the estimate of the derivatives of the cut-off function, which we had mistakenly stated to be independent of the order of the derivative. The second error was related to the scaling of the error when preparing the normalized state. Dr. Low's feedback has been invaluable in refining our approach, leading to a deeper understanding of the algorithm's optimality and allowing us to present an improved result. }

We kindly thank Dr. Guang Hao Low for pointing out to us some errors in the earlier version of this paper, where our earlier result gave a scaling that was near-optimal instead of optimal. This helped us to provide an improved result and to achieve optimality.

SJ and NL are supported by NSFC grant No. 12341104, the Shanghai Jiao Tong University 2030 Initiative and the Fundamental Research Funds for the Central Universities. SJ was also partially supported by the Shanghai Municipal Science and
Technology Major Project (2021SHZDZX0102). NL also acknowledges funding from NSFC grant No.12471411 and the Science and Technology Program of Shanghai, China (21JC1402900).
CM was partially supported by NSFC grant No. 12501607,
the Science and Technology Commission of Shanghai Municipality (No.22DZ2229014),
the China Postdoctoral Science Foundation (No. 2023M732248) and the Postdoctoral Innovative Talents Support Program (No. BX20230219).
YP was supported by the Xiangtan University 2025 Postgraduate Innovation Project (No. XDCX2025Y216).
YY was supported by NSFC grant No. 12301561, the Key Project of Scientific Research Project of Hunan Provincial Department of Education (No. 24A0100), the Science and Technology Innovation Program of Hunan Province (No. 2025RC3150) and the 111 Project (No. D23017).

%%\addcontentsline{toc}{section}{References}
\bibliographystyle{unsrt} %plain, unsrt, alpha
\bibliography{Refs}

\end{document}